\documentclass[reqno,12pt,letterpaper]{amsart}

\usepackage{amsmath,amssymb,amsthm,graphicx,mathrsfs,url}
\usepackage[usenames,dvipsnames]{color}
\usepackage{hyperref}
\hypersetup{pdfauthor={Name}}
\usepackage{amsxtra}
\usepackage{graphicx}
\usepackage{tikz}
\relax

\numberwithin{equation}{section}

\newcommand{\be}{\begin{equation}}
\newcommand{\ee}{\end{equation}}

\newcommand{\ov}{\overline}

\renewcommand{\dim}{\operatorname{dim}}

\renewcommand{\Re}{\mathop{\rm Re}\nolimits}
\renewcommand{\Im}{\mathop{\rm Im}\nolimits}

\newcommand{\dist}{\operatorname{dist}}

\theoremstyle{plain}

\newtheorem{thm}{Theorem}
\newtheorem{prop}{Proposition}[section]
\newtheorem{cor}[prop]{Corollary}
\newtheorem{lem}[prop]{Lemma}

\theoremstyle{definition}

\newtheorem{rem}[prop]{Remark}

\numberwithin{equation}{section}

\def\squarebox#1{\hbox to #1{\hfill\vbox to #1{\vfill}}}

\usepackage{amsxtra}

\ifx\pdfoutput\undefined
\DeclareGraphicsExtensions{.pstex, .eps}
\else
\ifx\pdfoutput\relax
\DeclareGraphicsExtensions{.pstex, .eps}
\else
\ifnum\pdfoutput>0
\DeclareGraphicsExtensions{.pdf}
\else
\DeclareGraphicsExtensions{.pstex, .eps}
\fi
\fi
\fi

\title[Time optimal  observability for the Grushin Schr\"odinger equation ]
{Time optimal observability for Grushin Schr\"odinger equation}
\author[N. Burq]{Nicolas Burq}
\address{Universit\'e Paris-Sud, Université Paris-Saclay and Institut Universitaire de France, Laboratoire de mathématiques d'Orsay, UMR 8628 du CNRS, B\^atiment 307, 91405 Orsay Cedex, France}
\email{nicolas.burq@math.u-psud.fr}
\author[C-M. Sun]{Chenmin Sun}
\address{Universit\'e de Cergy-Pontoise, Laboratoire de Mathématiques AGM, UMR  8088 du CNRS, 2 av. Adolphe Chauvin
95302 Cergy-Pontoise Cedex, France }
\email{chenmin.sun@u-cergy.fr}

\usepackage{amssymb}
\usepackage{amsmath, amsthm, amsopn, amsfonts}

\def\11{{\rm 1~\hspace{-1.4ex}l} }
\def\R{\mathbb R}
\def\C{\mathbb C}
\def\Z{\mathbb Z}
\def\N{\mathbb N}

\def\T{\mathbb T}

\begin{document}

	\begin{abstract}
		We consider two dimensional Grushin Schr\"odinger equation posed on a finite cylinder $\Omega=(-1,1)_x\times \T_y$ with Dirichlet boundary condition. We obtain the sharp observability by any horizontal strip, with the optimal time $T_*>0$ depending on the size of the strip. Consequently, we prove the exact controllability for Grushin Schr\"odinger equation. By exploiting the concentration of eigenfunctions of harmonic oscillator at $x=0$, we also show that the observability fails for any $T\leq T_*$. 
		
	\end{abstract}   
	
	\maketitle 
	\setlength{\parskip}{0.3em}  
	
	\section{Introduction}   
	\label{in}
	\subsection{Motivation}
	In this article we are interested in the observation and control for a sub-elliptic Schr\"odinger  equation on a finite cylinder. Observability for linear evolution PDEs has been extensively studied in the past decades. These studies aim to observe solutions of PDEs 
\begin{equation}\label{eq:abstract} \partial_tU-\mathcal{L}U=0,\quad U|_{t=0}=U_0
\end{equation}
	in a small region $\omega$ (control region) by spending some time. More precisely, it concerns an inequality of the type
	\begin{equation}\label{ob:abstract}
	 \|U(0)\|_{X}\leq C_T\int_0^T\|\mathbf{1}_{\omega}U(t)\|_X^2dt.
	\end{equation}
	If \eqref{ob:abstract} holds uniformly with respect to the initial data $U(0)$ in some Hilbert space $X$, then we say that the observability of the equation \eqref{eq:abstract} from $(0,T)\times\omega$ is  true.  
	There are two aspects. Firstly, observability is a quantitative version of the unique continuation property. The question of observability from $(0,T)\times\omega$ can be reformulated as: can sequences of solutions of \eqref{eq:abstract} concentrate on sets which do not intersect with $(0,T)\times\omega$? Secondly, it is equivalent to the null controllability (and hence for reversible systems to the exact controlability) of the dual equation of \eqref{eq:abstract}(see \cite{Li}). For hyperbolic PDEs, the observability can be obtained effectively by the theorem of propagation of singularities. In particular, for the wave equation, it is well-known that (see \cite{BLR},\cite{BG}) the observability is equivalent to the so called geometric control condition, which says that all the trajectories of the generalized geodesic flow will enter the control region before some prescribed time $T>0$. Moreover, one cannot observe the waves before this geometric time~$T$.
	
	For Schr\"odinger equation
	$$ i\partial_tu+\Delta u=0
	$$
	on  manifolds (with or without boundaries),
	due to the dispersion, or infinite propagation speed, it was shown by Lebeau (\cite{Le}) that the observability holds true for arbitrary short time $T$, provided that the control region $\omega$ satisfies the geometric control condition. This geometric restriction on $\omega$ is not necessary in general. It turns out that the stability(instability) of the geodesic flow of the underlying manifold plays an important role for the concentration of solutions of Schr\"odinger equation. It was shown in \cite{Ja},\cite{BZ4},\cite{BBZ},\cite{AM},\cite{ALM},\cite{Ji} that for any $T>0$ and any non-empty open set $\omega$, the observability of Schr\"odinger equation from $(0,T)\times\omega$ is true, if the underlying manifolds are torus, disk or hyperbolic surfaces. In this article, {in some sense,  we change dramatically the  geometry of the underlying manifold, and are interested in the observability problem for the Schr\"odinger equation in a  hypoelliptic geometry.} For our model example, a mixed feature of propagation and dispersion will be revealed, which will crucially influence the validity of the observability, leading to a unique example for which observability and control hold for the Schr\"odinger equation while it is violated for the heat equation~\cite{Ko17}.
	
	\subsection{Grushin Schr\"odinger equation} Denote by $\T:=\R/(2\pi\Z)$ the one-dimensional torus. We consider the following Grushin Schr\"odinger equation on a finite cylinder $\Omega=(-1,1)_x\times \T_y$ with Dirichlet boundary condition:
	\begin{equation}\label{main-equation-bounded}
	\begin{cases}
	&i\partial_tu+\Delta_Gu=0, (t,x,y)\in\mathbb{R}\times \Omega \\
	&u|_{t=0}=u_0 \in L^2(\Omega),\\
	&u|_{x=\pm 1}=0
	\end{cases}
	\end{equation}
	where the Grushin Laplace operator is defined as
	$ \Delta_G=\partial_x^2+x^2\partial_y^2
	$,
	together with domain
$$	D(\Delta_G):=\{u\in \mathcal{D}'(\Omega): \partial_x^2u, x^2\partial_y^2u\in L^2(\Omega)\text{ and } u\mid_{\partial \Omega} =0  \}.
$$
	The mass 
	$$ \|u(t)\|_{L^2(\Omega)}^2=\int_{\Omega}|u(t,x,y)|^2dxdy
	$$
	and the energy
	\begin{align}\label{H1G}
	 \|u\|_{\dot{H}_G^1(\Omega)}:=\int_{\Omega}(|\partial_xu(t,x,y)|^2+|x\partial_yu(t,x,y)|^2)dxdy
	\end{align}
	are conserved along the flow.

	The equation \eqref{main-equation-bounded} can be solved by first taking the Fourier transformation in $y$ variable, and then solving the spectral problem of 1D harmonic oscillator.  More precisely, for fixed $n\in \Z$, we denote by
	$\varphi_{m,n}(x)$ be the $m$-th eigenfunction of the operator $L_n=-\partial_x^2+n^2x^2$ with the domain $D(L_n)=H^2((-1,1))\cap H_0^1((-1,1))$,  associated with eigenvalues $\lambda_{m,n}^2$. We expand the initial data $u_0(x,y)$ as
	$$ u_0(x,y)=\sum_{m=0}^{\infty}\sum_{n\in\Z}a_{m,n}\varphi_{m,n}(x)e^{iny},
	$$
	hence, the solution of \eqref{main-equation-bounded} is given by
	$$ (e^{it\Delta_G}u_0)(t,x,y)=\sum_{m=0}^{\infty}\sum_{n\in\Z}a_{m,n}e^{-it\lambda_{m,n}^2}\varphi_{m,n}(x)e^{iny}.
	$$ 
	In particular,
	\begin{itemize}
		\item $\mathrm{Spec}(-\Delta_G)=\{\lambda_{m,n}^2:m\in\mathbb{N},n\in\mathbb{Z}\}$.
		\item Eigenfunctions of $-\Delta_G$ are
		$$ \Phi_{m,n}(x,y)=\varphi_{m,n}(x)e^{iny}.
		$$
	\end{itemize}

	A classical question in control theory is whether,  given a non-empty open set $\omega$ and time $T_0>0$, for $u_0\in L^2(\Omega)$,  there exists a
	control $f\in L^2((0,T)\times\omega)$ such that the solution of (1.1) satisfies $u|_{t\geq T_0}=0$.
	
	 Control and observation problem for the degenerate parabolic equation has attracted lots of attentions in these years. Many results have been obtained, yet, the picture is far from been completed. For the parabolic Grushin equation
	 $$ \partial_tu-\Delta_Gu=0,
	 $$ 
	 it turns out that there exists a minimal time $T^*>0$ for the observability, if $\omega$ does not touch the line $x=0$ on which $\Delta_G$ is degenerate (see \cite{BCG14}). It was also shown in \cite{Ko17} that the observability is untrue, if $\omega$ does not intersect a horizontal strip. Though governed by different mechanism, we present a time-optimal control and observation result for the Grushin Schr\"odinger equation, as a counterpart of the mentioned degenerate parabolic equations. 	
	    
	\subsection{Main results}	
	In this article, we consider the horizontal control domain $\omega$ of the form $(-1,1)_x\times I$, where $I\subset \T$ is a finite union of intervals. For such $\omega$, we define the quantity $\mathcal{L}(\omega)$ to be the length of the largest interval in $\Omega\setminus\omega\cap \{ x=0\}$:
	$$ \mathcal{L}(\omega):=\sup\{s: \exists y_1,y_2\in\T \textrm{ such that } \dist_{\T}(y_1,y_2)=s,\textrm{ and } [(0,y_1),(0,y_2) ]\cap \omega=\emptyset  \}.
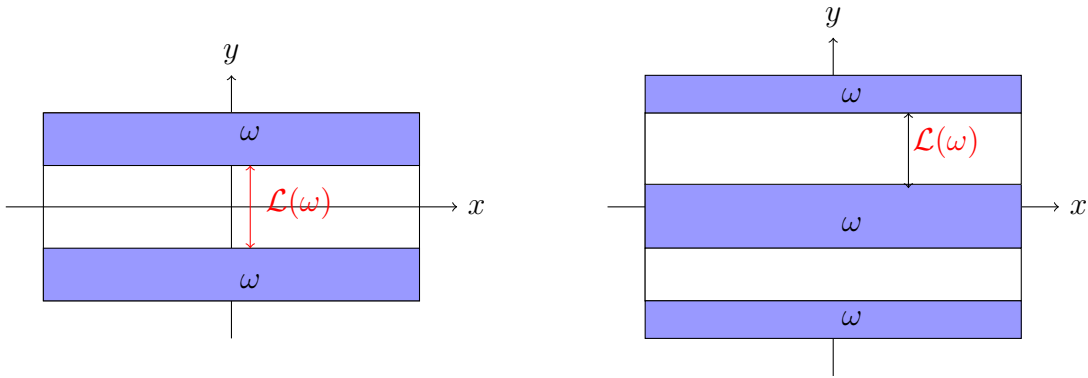
	$$ 
	\begin{center}
		\begin{tikzpicture}
		\draw[->] (-3,0) -- (3,0);
		\draw (3,0) node[right] {$x$};
		\draw [->] (0,-1.75) -- (0,1.75);
		\draw (0,1.75) node[above] {$y$};
		\draw (-2.5,-1.25) rectangle (2.5,1.25);
		\draw[fill=orange!40] (-2.5,1.25) rectangle (2.5,1.25);
		\draw[fill=blue!40] (-2.5,0.55) rectangle (2.5,1.25);
		\draw[fill=blue!40] (-2.5,-1.25) rectangle (2.5,-0.55);
		\draw[black] (0.25,0.75) node[above] {$\omega$};
		\draw[black] (0.25,-0.75) node[below] {$\omega$};
		\draw[red] [<->] (0.25,-0.55) -- (0.25,0.55);
		\draw[red] (0.9,-0.3) node[above] {$ \mathcal{L} (\omega)$};
		\draw[->] (5,0) -- (11,0);
		\draw (11,0) node[right] {$x$};
		\draw [->] (8,-2.25) -- (8,2.25);
		\draw (8,2.25) node[above] {$y$};
		\draw (5.5,-1.25) rectangle (10.5,1.25);
		\draw[fill=blue!40] (5.5,-0.55) rectangle (10.5,0.3);
		\draw[fill=white!40] (5.5,0.3) rectangle (10.5,1.25);
		\draw[fill=white!40] (5.5,-1.25) rectangle (10.5,-0.55);
		\draw[fill=blue!40] (5.5, 1.25) rectangle (10.5,1.75);
		\draw[fill=blue!40] (5.5,-1.75) rectangle (10.5,-1.25);
		\draw [<->] (9,0.25) -- (9,1.25);
		\draw[red] (9.5,0.5) node[above] {$ \mathcal{L} (\omega)$};
		\draw[black] (8.25,1.25) node[above] {$\omega$};
		\draw[black] (8.25,-1.25) node[below] {$\omega$};
		\draw[black] (08.25,0) node[below]{$\omega$};
		\end{tikzpicture}
	\end{center}
	We show that if $T_0>\mathcal{L}(\omega)$, the observability is true:
		\begin{thm}\label{positive-bounded}
		Let $\omega$ be a horizontal strip of $\Omega$ and assume that $T_0>\mathcal{L}(\omega)$. Then there exists $ C_{T_0}>0$, such that for any solution $u$ to \eqref{main-equation-bounded} with initial data $u_0\in L^2(\Omega)$, the observability
		\begin{equation}\label{thm-observability}
		\|u_0\|_{L^2(\Omega)}^2\leq C_{T_0}\int_{0}^{T_0}\|u(t)\|_{L^2(\omega)}^2dt 
		\end{equation}
		holds true.
	\end{thm}
	Consequently, by the standard HUM method (\cite{Li}, see for example \cite{BZ4} in the context of Schr\"odinger equation), we obtain the following exact controllability:
	\begin{thm}\label{control:exact}
		Assume that $T_0>\mathcal{L}(\omega)$, then for any $u_0\in L^2(\Omega)$, there exists $g\in L^2((0,T_0)\times \omega)$, such that the solution of the equation
		$$ i\partial_tu+\Delta_Gu=\mathbf{1}_{\omega}g,\quad u|_{t=0}=u_0
		$$
		satisfies $u(T_0,\cdot)=0$.
	\end{thm}
	Our next result shows that $T_0>\mathcal{L}(\omega)$ is also necessary for the observability. This will be done by constructing examples which contradicts the observability, if $T_0\leq \mathcal{L}(\omega)$. The precise statement is as follows:
	\begin{thm}\label{thm-negative}
		Assume that $T_0\leq \mathcal{L}(\omega)$, then there exists a sequence $(u_{k,0})\subset L^2(\Omega)$, such that 
		$$ \|u_{k,0}\|_{L^2(\Omega)}=1, \textrm{ and }
		\lim_{k\rightarrow\infty}\int_{0}^{T_0}\|e^{it\Delta_G}u_{k,0}\|_{L^2(\omega)}^2dt=0.
		$$
	\end{thm}
	 The observability for Schr\"odinger equation is related to the long-time behaviour of the semi-classical Schr\"odinger equation. Starting from initial data oscillating at  space frequency of order $h^{-1}$, a simple time rescalling gives
	 $$ i \partial_t u + \Delta u =0 \Leftrightarrow  ih\partial_sv+h^2\Delta v=0, v(s,x) = u(hs, x).$$

	 If the geometric control condition is fullfilled, the concentration outside $\omega$ cannot occur beyond the semi-classical time scale $s\sim 1$. This in turn tells us $t\sim h$, which is responsible for the short time observability result of Lebeau \cite{Le}. In the instable cases (torus, disk), we can see dispersion by understanding propagation beyond the semi-classical time scale $s\sim \epsilon h^{-1}$, $\epsilon\ll 1$, and this is responsible for the observability results in \cite{BZ4} \cite{AM},\cite{ALM} at the classical time scale $\epsilon$. 
	 To the best of our knowledge, Theorem \ref{positive-bounded} is the first result for the Schr\"odinger type equations where we can track dispersion all the way up to the  exact semi-classical time scale $s=\mathcal{L}(\omega)h^{-1}$.

	\begin{rem}
		The method we used in the proof of Theorem \ref{control:exact}, Theorem \ref{thm-observability} and Theorem \ref{thm-negative} would allow one to obtain similar results for more general hypoelliptic operator
		of the form
		$$ P=\partial_x^2+V(x)\partial_y^2,
		$$
		where $V(x)\in C^{\infty}([-1,1])$, satisfying
		$$ V(x)>0,\textrm{ if }x\neq 0, \quad V(x)\sim x^2, V'(x)\sim x, V''(x)\sim 1,x\rightarrow 0. 
		$$
		A typical example is $V(x)=\sin^2(\pi x)$.
	\end{rem}
	
	\begin{rem}
		Though the control domain $\omega$ we considered here is horizontal strips, with a little effort, one can prove the exact controllability with the same time threshold $T_0>\mathcal{L}(\omega)$, when $\omega$ contains a horizontal strip. Moreover, from the concentration of the first eigenfunction of the semi-classical harmonic oscillator $-\partial_x^2+n^2x^2$, the exact controllability is untrue, for any $T>0$, if $\Omega\setminus\omega$ contains a neighborhood of the degenerate line $\{x=0 \}$. It would be interesting to know whether the time-optimal observability (exact controllability) is true, if $\omega$ is only a neighborhood of some verticle interval $[(0,y_1), (0,y_2) ]$.
	\end{rem}
	
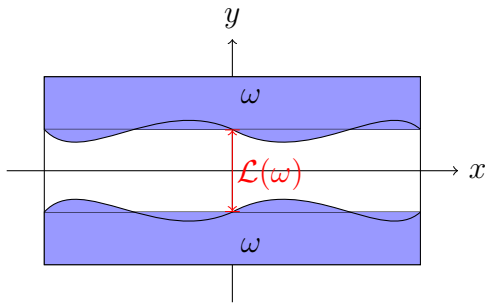
\begin{center}
		\begin{tikzpicture}
		\draw[->] (-3,0) -- (3,0);
		\draw (3,0) node[right] {$x$};
		\draw [->] (0,-1.75) -- (0,1.75);
		\draw (0,1.75) node[above] {$y$};
		\draw (-2.5,-1.25) rectangle (2.5,1.25);
		\draw[fill=blue!40] (-2.5,0.55) rectangle (2.5,1.25);
		\draw[fill=blue!40] (-2.5,-1.25) rectangle (2.5,-0.55);
		\draw[black] (0.25,0.75) node[above] {$\omega$};
		\draw[black] (0.25,-0.75) node[below] {$\omega$};
		\draw[fill=blue!40] (-2.5,0.55) .. controls (-2,0) and (-1,1) .. (0,0.55)
		..controls (1,0) and (2,1) .. (2.5,0.55);
		\draw[fill=blue!40] (-2.5,-0.55) .. controls (-2,0) and (-1,-1) .. (0,-0.55)
		..controls (1,0) and (2,-1) .. (2.5,-0.55);
		\draw[red] [|<->|] (0,-0.55) -- (0,0.55);
		\draw[red] (0.5,-0.4) node[above] {$ \mathcal{L} (\omega)$};
		\end{tikzpicture}
	\end{center}
	
	\begin{rem}
	It has been shown recently (see \cite{Let}) that subelliptic wave equations are never observable. Though the non-observability in \cite{Let} is proved for the general sub-Riemannian setting, in the special case of the Grushin-wave equation, it is possible to provide a more straightforward proof. As a comparison to the wave-type equation, in Appendix, we present a direct proof of Theorem 1 in \cite{Let} for the Grushin-wave equation.  
	\end{rem}
	
	\subsection*{Acknwoledgement}
	The first author is supported by Institut Universitaire de France and ANR  grant ISDEEC, ANR-16-CE40-0013,  and the second author is supported by ANR grant ODA (ANR-18-CE40-
	0020-01). The problem considered here was inspired by a talk of Armand Koenig. The authors would like to thank Richard Lascar, for his interest, especially for pointing out the reference \cite{Melrose1}. We are grateful to Fabricio Maci\`a for a discussion while the second author attended the Journ\'ees EDPist 2019, Obernai.

	\section{Reformulation of the problem and outline of the proofs}
 Assume that $\omega=(-1,1)_x\times \cup_{k=1}^N (c_k,d_k)$, where $(c_k,d_k)$ are disjoint intervals of $\T$. For $\alpha_0>0$, define the smooth function $\phi\in C^{\infty}(\T)$ by	
		\begin{equation}\label{phi}
	\phi(y)=\sum_{j=1}^N\phi_j(y),\quad\text{ where } \phi_j(y) = \begin{cases} &1 \text{ if } y\in (c_j,d_j) \textrm{ and } \text{dist}(y,\frac{c_j+d_j}{2})\leq \frac{|d_j-c_j|}{2}-\alpha_0 \\
	&0 \text{ otherwise }.
	\end{cases}
	\end{equation}
Since $T_0>\mathcal{L}(\omega)$, for 
$$ \alpha_0<\min\big\{|d_k-c_k|/2,(T_0-\mathcal{L}(\omega))/2; k=1,2\cdots,N \big\},
$$
we denote by $a_0=\frac{\mathcal{L}(\omega)}{2}+\alpha_0$. Hence we reformulate Theorem \ref{positive-bounded} as follows:
	\begin{thm}\label{prop-thm1}
		Assume that $T>a_0$, then for any $u_0\in L^2(\Omega)$, we have
		\begin{equation}\label{propo-thm1-1}
		\|u_0\|_{L^2(\Omega)}^2\leq C_{T,a_0}\int_{-T}^T\|\phi(y)e^{it\Delta_G}u_0\|_{L^2(\Omega)}^2dt.
		\end{equation}	
	\end{thm}

	Throughout this article, we will use the standard notation
	$D_y=\frac{1}{i}\partial_y, D_x=\frac{1}{i}\partial_x$.
	\subsection{Outline of the non-observability for \texorpdfstring{$T_0\leq \mathcal{L}(\omega)$}{T0< L(omega)}}
	Unlike the classical Schr\"odinger equation, the Grushin-Schr\"odinger equation exhibits the features of both wave equation and dispersive equations. To understand the heuristics, consider Grushin-Schr\"odinger equation posed on $(x,y)\in\R^2$:
	$$  i\partial_tu+\Delta_Gu=0.
	$$ 
	The solution formula can be written explicitly 
	\begin{equation}\label{Grusin-free}
	u(t,x,y)=\frac{1}{\sqrt{2\pi}}\sum_{m=0}^{\infty}\int_{\R}e^{-it(2m+1)|\eta|+iy\eta}\widehat{f}_{m}(\eta)h_m(\sqrt{|\eta|}x)d\eta,
	\end{equation}
	where $h_m(x)$ is the $m$-th Hermite function, eigenfunction of the harmonic oscillator $-\partial_x^2+x^2$. As explained in \cite{Ge}, there exists an orthogonal decomposition
	$$ L^2(\R^2)=\oplus_{\pm}\oplus_{m=0}^{\infty}V_m^{\pm},\quad \Delta_G|_{V_m^{\pm}}=\pm i(2m+1)\partial_y,
	$$
	hence \eqref{Grusin-free} can be viewed as a coupled system of transport equations:
	$$ i(\partial_t\pm (2m+1)\partial_y)u_m^{\pm}=0.
	$$
	The $m$-th wave $u_m^{\pm}$ propagates on the vertical direction of velocity $2m+1$. Actually, the bottom spectral portion $m=0$ is responsible for the restriction $T_0>\mathcal{L}(\omega)$ for the observability. For larger $m$, the propagation speed of the portion $u_m^{\pm}$ is larger, and this can be interpreted as ``dispersion". In the section 9, we will construct concrete example to disprove the observability inequality if $T_0\leq \mathcal{L}(\omega)$. Unlike the equation \eqref{Grusin-free} posed on the whole space, we have no explicit formula for the first eigenfunction of the semi-classical harmonic oscillator $-\partial_x^2+\eta^2x^2$ with Dirichlet boundary condition. The construction is then based on the careful quantitative analysis of the first eigenvalue and eigenfunction of the semi-classical harmonic oscillator.

	\subsection{Outline of the proof of Observability if \texorpdfstring{$T_0>\mathcal{L}(\omega)$}{T0> L(omega)}}
	For simplicity, we shall first study the case where there is a single band in $\omega^c$, and by translation in $y$ we can assume that $\omega^c=(-1,1)_x\times (-a,a)$, hence $\mathcal{L}(\omega)=2a$, and $a_0\in (a,T_0/2)$. 
	Assume that $\psi\in C_c^{\infty}(\R;[0,1])$ such that
	$$ \psi(\zeta)\equiv 1,\textrm{ if } \frac{1}{2}\leq |\zeta|\leq 2 \textrm{ and } \psi(\zeta)\equiv 0
	\textrm{ if }|\zeta|<\frac{1}{4} \textrm{ or } |\zeta|>4. $$ From standard compactness argument, one can reduce the proof of Proposition \ref{prop-thm1} to the following spectral-localized version:
	\begin{prop}\label{ob:spectral-localized}
		Let $T>a_0$, then there exists $C_{T,a_0}>0$ such that for any $u_0\in L^2(\Omega)$ and all $0<h\ll 1$, we have
		$$ \|\psi(h^2\Delta_G)u_0\|_{L^2(\Omega)}^2\leq C_{T,a_0}\int_{-T}^{T}\|\phi(y)e^{it\Delta_G}\psi(h^2\Delta_G)u_0\|_{L^2(\Omega)}^2dt.
		$$ 
	\end{prop}
	Due to the anisotropic feature, we need to analyze different regimes. A basic observation is that $\partial_y$ commutes with $-\Delta_G$. When the spectrum of $\sqrt{-\Delta_G}$ is localized at the scale $ h^{-1}$, by the hypoellipticity, the vertical frequency $|D_y|$ is contrained to $|D_y|\lesssim h^{-2}$. In Section 3, we will prove this simple coercivity estimate. In the sequel, our  analysis will require a second microlocalization with respect to the variable $D_y$, which will in turn require a second microlocalization with respect to the variables $(x, D_x)$. Fortunately this will remain at a quite basic level  and appear only through the elementary elliptic regularity result Proposition~\ref{elliptique}.
	
	In section 4, we will deal with the regime:
	$$ \sqrt{-\Delta_G}=\sqrt{D_x^2+x^2D_y^2}\sim h^{-1}\ll |D_y|\lesssim h^{-2}.
	$$
	We introduce another scale $|D_y|\sim \hbar^{-1}$, and the half wave regime corresponds to  such that  $h^2\lesssim \hbar \ll h$. The main constranits in this regime come from the subregime (proper half-wave regime) when $\hbar\sim h^{2}$, and $|x|\lesssim \hbar^{1/2}$, since $|x|\gg \hbar^{1/2}$ is the classical forbidden region where the contributions are of order $O(h^{\infty})$, if we write
	$$ i\partial_t+\Delta_G=i\partial_t+\partial_x^2+\underbrace{ (\frac{1}{i}\hbar \partial_y)(\hbar^{-1/2} x)^2 }_{\textrm{uniformly bounded operator}}i\partial_y.
	$$
	
	It turns out that the operator $\frac{1}{i}\hbar\partial_y (\hbar^{-1/2} x)^2$ is also bounded from below, in an averaged sense. In this regime, the key tool will be the positive commutator relation
	$$ [-\Delta_G, x\partial_x + 2 y \partial_y] = -2 \Delta _G
	$$
	 Roughly speaking, the commutator relation 
	$$[i\partial_t+\Delta_G, \varphi_T(t)(x\partial_x+2y\partial_y) ]=2\Delta_G+i\varphi_T'(t)(x\partial_x+2y\partial_y)
	$$
	 allows one to control the energy norm $2T\times 2\big(\|x\partial_yu \|_{L^2}^2+\|\partial_xu\|_{L^2}^2\big)$ by the contributions inside the control region $\omega$ and $(\varphi_T'(t)2y\partial_yu,u)$. The difficulty is the fact that in the half-wave regime, $|x|\lesssim \hbar^{1/2}$, one cannot absorb the norm $\|\partial_yu\|_{L^2}$ by the energy norm. To overcome this difficulty, we observe that, heuristically, at $t=0$, the solutions are concentrated near $y=0$, then on the support of $\varphi_T'(t)$, $t\sim T$, the solutions should be propagated near $|y|=T>a$. With appropriate time cut-off $\varphi_T(t)$, this term $|(\varphi_T'(t)2y\partial_yu,u)|$ is essentially $$2a\times 2\times\hbar^{-1}\times \|u\|_{L^2}^2+\textrm{ errors }.$$
	By using the hypoellipticity, for the well-localized solution $u$, the main contribution turns out to be $$4a\times \big(\|x\partial_yu\|_{L^2}^2+\|\partial_xu\|_{L^2}^2 \big).$$ This allows us to conclude, provided that $T>a$.
	
	In the section 5, we discuss the semi-classical dispersive regime, in which the geometric control condition is fullfilled:
	$$ \sqrt{-\Delta_G}\sim |D_y|\sim h^{-1}.
	$$
	From the time rescaling trick introduced in \cite{Le}, it is then reduced to prove the observability estimate for the semi-classical Grushin Schr\"odinger equation
	\begin{equation}\label{semi-GCC} ih\partial_sw+(h\partial_x)^2w+x^2(h\partial_y^2)w=0.
	\end{equation}
	In the considered regime, the Hamiltonian flow is never degenerated and all geodesics enter the control region in finite time.  By a standard defect-measure based argument, we are able to prove the semi-classical propagation observability for \eqref{semi-GCC}. 
	
	In the sections 6 and 7, we will deal with the non-semiclassical dispersive regime. If $$h^{\epsilon}\lesssim |D_y|\ll \sqrt{-\Delta_G}\sim h^{-1},
	$$
	thanks to the rapid propagation in the horizontal variable $x$, we are still able to use the positive commutator method to detect the vertical propagataion, but in this case the arguments are simpler and any $T>0$ would work. Finally, if
	$$ |D_y|\ll h^{\epsilon}\ll \sqrt{-\Delta_G}\sim h^{-1},
	$$
	we will use a normal form transformation method, inspired by the work of the first author and M.~Zworski \cite{BZ4}, to convert $i\partial_t+\Delta_G$ to
	$ i\partial_t+\partial_x^2+M^2\partial_y^2
	$
	in this regime, modulo errors. Then an application of the observality theorem for the classical Schr\"odinger equation allows us to conclude.

\section{Coercivity estimate}
From hypoellipticity, we have the following simple coercivity estimate:
\begin{lem}\label{coercivity}
	For every $f\in D(\Delta_G)$, there holds
	$$ \||D_y|^{1/2}f\|_{L^2(\Omega))}^2\leq (-\Delta_Gf,f)_{L^2(\Omega)}.
	$$ 	
\end{lem}
\begin{proof}
	By doing integration by part, we have
	$$(-\Delta_Gf,f)_{L^2(\Omega)}^2=\|\partial_xf\|_{L^2(\Omega)}^2+\|x\partial_yf\|_{L^2(\Omega)}^2.
	$$
	Denote by $\Pi_{+}, \Pi_-$, the projection to the strictly positive and the strictly negative frequencies in $y$ variable. Obviously, $\Pi_{\pm}$ commute with $\Delta_G, \partial_x, x\partial_y$. From the commutator relation $[\partial_x,x\partial_y]=\partial_y$, we have
	\begin{equation*}
	\begin{split}
	\||D_y|^{1/2}\Pi_{\pm}f\|_{L^2(\Omega)}^2=\mp i&(\partial_y\Pi_{\pm}f,\Pi_{\pm}f )_{L^2(\Omega)}=\mp i([\partial_x,x\partial_y]\Pi_{\pm}f,\Pi_{\pm}f )_{L^2(\Omega)}.
	\end{split}
	\end{equation*}
	After integration by part, we have
	$$ \||D_y|^{1/2}\Pi_{\pm}f\|_{L^2(\Omega)}^2= \pm2\Im(\partial_x\Pi_{\pm}f, x\partial_y\Pi_{\pm}f )_{L^2(\Omega)}.
	$$
	Thus
	$$ \||D_y|^{1/2}\Pi_{\pm}f\|_{L^2(\Omega)}^2\leq \|\partial_x\Pi_{\pm}f\|_{L^2(\Omega)}^2+\|x\partial_y\Pi_{\pm}f\|_{L^2(\Omega)}^2.
	$$
	The proof of Lemma \ref{coercivity} is now complete.
\end{proof}
Let $\chi_0\in C_c^{\infty}(\R;[0,1])$ such that 
$$ \chi_0(\zeta)\equiv 1,\textrm{ if }|\zeta|\leq 2 \textrm{ and } \chi_0(\zeta)\equiv 0 , \textrm{ if } |\zeta|>3.
$$
\begin{cor}\label{regime:elliptic}
	For $0<h<1$ and $e_0> \frac{1}{2}$,
	$$ \psi(h^2\Delta_G)\left(1-\chi_0(e_0h^2D_y) \right)=0.
	$$
\end{cor}
\begin{proof}
	Take $f\in D(\Delta_G)$, and we expand $f_h:=\psi(h^2\Delta_G)\left(1-\chi_0(e_0h^2D_y)\right)f$  $$f_h=\sum_{m,n}a_{m,n}\psi(h^2\lambda_{m,n}^2)(1-\chi_0(e_0h^2n) )\varphi_{m,n}e^{iny}.$$
	Applying Lemma \ref{coercivity} to $f_h$, we obtain that
	\begin{equation*}
	\begin{split}
	&\sum_{m,n}|n||a_{m,n}|^2\psi(h^2\lambda_{m,n}^2)^2\left(1-\chi_0(e_0h^2n)\right)^2\leq \sum_{m,n}\lambda_{m,n}^2|a_{m,n}|^2\psi(h^2\lambda_{m,n}^2)^2\left(1-\chi_0(e_0h^2n)\right)^2.
	\end{split}
	\end{equation*}
	For the non-zero contributions on the right side, $\lambda_{m,n}^2\leq \frac{4}{h^2}$, while for the non-zero contributions on the left side, we have $|n|>\frac{2}{e_0h^2}$. Therefore, if $e_0<\frac{1}{2}$, the only possibility is that $f_h=0$. This completes the proof of Corollary \ref{regime:elliptic}.
	
\end{proof}

\section{Half-wave regime:}   
In this section, the numerical constant $e_0>\frac{1}{2}$ is fixed, and the cut-offs $\chi_0$ and $\psi$ are also fixed as in the previous section. Denote by $\widetilde{\psi}$ another cut-off function with similar support property as $\psi$, such that
$$  \psi \widetilde{\psi}\equiv \psi.
$$
For $0<h<1, b_0>0$, we define the semi-classical spectral projector
$$ \Pi_{h}^{b_0h}:=\psi(-h^2\Delta_G)\left(\chi_0(e_0h^2D_y)-\chi_0(b_0hD_y)\right).
$$
This allows us to localize the spectral into the regime: $h^{-1}\ll|D_y|\lesssim h^{-2}$, and the half-wave regime formally corresponds to
$h^2|D_y|\sim 1, h\sqrt{-\Delta_G}\sim 1.$

The goal of this section is to prove the following observability estimate:
\begin{prop}\label{regime:Half-wave}
	Fix $T>a_0$. There exist $b_0>0$,   $h_0>0$, such that for all $h<h_0$  we have 
	\begin{equation}\label{regime:demi-ondes}
	\begin{split}
	\|\Pi_{h}^{b_0h}u_0\|_{L^2(\Omega)}^2\leq C_{T,a}\int_{-T}^{T}\|\phi(y)e^{it\Delta_G}\Pi_{h}^{b_0h}u_{0}\|_{L^2(\Omega)}^2dt.
	\end{split}
	\end{equation}
\end{prop}
We will prove the inequality above by proving firstly a localized version for $h^{-1}\ll |D_y|\sim \hbar^{-1}\leq 3(e_0h)^{-2}$, and gluing such inequalities together. The constant $b_0>0$ will be chosen in the proof as a small parameter. Note that $\hbar$ can be viewed as a second semi-classical parameter.

For any $\hbar^{-1}>(b_0h)^{-1}$, $0<\delta<1$, we denote by
\begin{equation}\label{localization-second}
 u_{h,\hbar,\delta}:=\widetilde{\psi}(-h^2\Delta_G)\widetilde{\psi}\left(\frac{\hbar|D_y|-1}{\delta}\right)u.
\end{equation}

\subsection{Elliptic regularity}

Denote by $\mathcal{L}_n=-\partial_x^2+n^2x^2$, the semi-classical harmonic oscillator associated with the Dirichlet boundary condition. Note that $\varphi_{m,n}(x)$ is the $m$ th eigenfunction of $\mathcal{L}_n$ with eigenvalue $\lambda_{m,n}^2$. Denote by $H_n=-\partial_x^2+n^2x^2$, the semi-classical harmonic oscillator on the whole line with domain $D(H_n)=H^2(\R)\cap \mathcal{H}^1(\R)$, where $\mathcal{H}^1(\R)$ is defined via the norm
$$ \|f\|_{\mathcal{H}^1(\R)}^2:=\|\partial_xf\|_{L^2(\R)}^2+\|xf\|_{L^2(\R)}^2.
$$
Denote by $\omega_{j,n}^2$, the $j$ th eigenvalue of $H_{n}$. It is well-known that
$ \omega_{j,n}^2=(2j+1)|n|.
$
We need the following comparison principle.
\begin{lem}\label{comparison}
	Let $n\in\N$, then for any $j\in\N$, we have
	$$ \omega_{j,n}^2 \leq \lambda_{j,n}^2.
	$$	
\end{lem}
\begin{proof}
	This is a simple consequence of the max-min principle. Recall that from the max-min principle, 
	$$ \lambda_{j,n}^2=\sup_{\psi_1,\cdots,\psi_{j-1}\in L^2((-1,1))}\inf_{\substack{f\in\textrm{ span }(\psi_1,\cdots,\psi_{j-1})^{\perp }\\
			f\neq 0, f\in H_0^1((-1,1)) }   }\frac{( \mathcal{L}_{n}f,f)_{L^2((-1,1))}}{\|f\|_{L^2((-1,1))}^2},
	$$	
	and
	$$\omega_{j,n}^2=\sup_{\psi_1,\cdots,\psi_{j-1}\in L^2(\R)}\inf_{\substack{f\in\textrm{ span }(\psi_1,\cdots,\psi_{j-1})^{\perp }\\
			f\neq 0, f\in \mathcal{H}^1(\R) }   }\frac{( H_{n}f,f)_{L^2(\R)}}{\|f\|_{L^2(\R)}^2}.
	$$
	Define $\iota f:=f\mathbf{1}_{(-1,1)}$, the zero-extension of $f\in H_0^1((-1,1))$, then obviously $\iota f\in \mathcal{H}^1(\R)$. Thus by definition, we have
	\begin{equation*}
	\begin{split}
	\omega_{j,n}^2\leq &\sup_{\psi_1,\cdots,\psi_{j-1}\in L^2(\R)}\inf_{\substack{\iota f\in\textrm{ span }(\psi_1,\cdots,\psi_{j-1})^{\perp }\\
			f\neq 0, f\in H_0^1((-1,1)) }   }\frac{( H_{n}(\iota f),\iota f)_{L^2(\R)}}{\|\iota f\|_{L^2(\R)}^2}\\
	=&\sup_{\psi_1,\cdots,\psi_{j-1}\in L^2(\R)}\inf_{\substack{\iota f\in\textrm{ span }(\psi_1,\cdots,\psi_{j-1})^{\perp }\\
			f\neq 0, f\in H_0^1((-1,1)) }   }\frac{( \mathcal{L}_{n}f,f)_{L^2(\R)}}{\|f\|_{L^2(\R)}^2}\\
	\leq & \sup_{\psi_1,\cdots,\psi_{j-1}\in L^2((-1,1))}\inf_{\substack{f\in\textrm{ span }(\psi_1,\cdots,\psi_{j-1})^{\perp }\\
			f\neq 0, f\in H_0^1((-1,1)) }   }\frac{( \mathcal{L}_{n}f,f)_{L^2((-1,1))}}{\|f\|_{L^2((-1,1))}^2}\\
	=&\lambda_{j,n}^2.
	\end{split}
	\end{equation*}
	This completes the proof of Lemma \ref{comparison}.  
\end{proof}

Using that by comparison principle,
we get:
\begin{cor}\label{Weyl}
For fixed $n$,
$$ N_{n,\tau}:=\#\{j: \lambda_{j,n}^2\leq \tau^2 \}\leq \frac{\tau^2}{2|n|}.
$$
\end{cor}

Let $\chi_1\in C_c^{\infty}(\R)$ such that $\chi_1(\zeta)$ vanishes if $|\zeta|>C_1$ or $|\zeta|<c_1$, for some $C_1>c_1>0$. We have the following
\begin{prop}[Elliptic regularity]\label{elliptique}
Let $0<\sigma<1$. There exist small constants $0<h_0\ll 1, 0<b_0\ll 1$, such that the for all $0<h<h_0$ and $\frac{e_0h^2}{4}<\hbar <b_0h$, the following statement is true. For $w_{h,\hbar}=\psi(h^2\Delta_G)\chi_1(\hbar D_y)w_{h,\hbar}$ and 
 $\chi\in C_c^{\infty}(-1,1)$ such that $\chi(x)\equiv 1$ near $0$,
we have with $\epsilon = \max (h^{1- \sigma}, 5 c_1^{-1} \hbar h^{-1})$
\begin{equation}\label{eq.elliptic}
 \|(1-\chi)(\frac{x} \epsilon )h\partial_xw_{h,\hbar}\|_{L^2(\Omega)}+ \|(1-\chi)(\frac{x} \epsilon )w_{h,\hbar}\|_{L^2(\Omega)}\leq C_N h^N,
\end{equation}
for all $N\in\N$.
\end{prop}

\begin{proof}
We first prove the elliptic estimate for eigenfunctions:
\begin{lem}For all $\chi \in C^\infty(-1,1)$ and all $N$, there exists $C>0$ such that for all eigenfunction $\Phi_{m,n}(x,y)=e^{in y}\varphi_{m,n}(x)$ satisfying $|n|\in\left[c_1\hbar^{-1},C_1\hbar^{-1}\right],  h^2\lambda_{m,n}^2\in\mathrm{supp}(\psi)$, and $\| \Phi_{m,n}\|_{L^2(-1,1)} =1$, the estimate
\begin{equation}\label{eq.elliptic2} \|(1-\chi)(\frac{x} \epsilon )h\partial_x\Phi_{m,n}\|^2_{L^2(\Omega)}+\|(1-\chi)(\frac{x} \epsilon )\Phi_{m,n}\|^2_{L^2(\Omega)}\leq C_N h^N
\end{equation}
holds for all $N\in\N$.
\end{lem}
 \begin{proof}
  The eigenfunction $\varphi_{m,n}$ of the operator $\mathcal{L}_n=-\partial_x^2+n^2x^2$ satisfies
\begin{equation}\label{semi-harmonic}
\mathcal{L}_n\varphi_{m,n}=\lambda_{m,n}^2\varphi_{m,n},\; \lambda_{m,n}^2\in \left[\frac{1}{16h^2},\frac{16}{h^2} \right],\quad \varphi_{m,n}|_{x=\pm 1}=0.
\end{equation}
Taking the Fourier transform with respect to the $y$ variable, multiplying \eqref{semi-harmonic} by $(1-\chi)^2(\frac x \epsilon) \ov{\varphi}_{m,n}$ and doing the integration by part, we obtain that
\begin{multline} \int_{-1}^1(1-\chi)^2( \frac x \epsilon) \left[|\varphi'_{m,n}|^2+n^2x^2|\varphi_{m,n}|^2 \right]dx-\int_{-1}^1 \frac 2 \epsilon (1-\chi)\chi' (\frac x \epsilon)\varphi_{m,n}'\ov{\varphi}_{m,n}dx\\
=\lambda_{m,n}^2\int_{-1}^1(1-\chi)^2( \frac x \epsilon)|\varphi_{m,n}|^2dx.
\end{multline}
This yields
$$ \int_{-1}^1(1-\chi)^2( \frac x \epsilon) \left[|\varphi_{m,n}'|^2+(n^2x^2-\lambda_{m,n}^2)|\varphi_{m,n}|^2 \right]dx\leq \frac 2 \epsilon \int_{-1}^1 |(1-\chi)\chi' \varphi_{m,n}'\varphi_{m,n}|dx.
$$
Note that on the support of $(1-\chi)( \frac x \epsilon)$, $n^2x^2-\lambda_{m,n}^2\geq 9h^{-2}$, which implies
\begin{multline}\label{eq.elliptic3}
 h^2\|(1-\chi)( \frac x \epsilon) \varphi_{m,n}'\|_{L^2(-1,1)}^2+9\|(1-\chi)( \frac x \epsilon) \varphi_{m,n}\|_{L^2(-1,1)}^2\\
 \leq \frac {2h^2} \epsilon \|(1-\chi)( \frac x \epsilon) \varphi_{m,n}'\|_{L^2(-1,1)}\|\chi' ( \frac x \epsilon)\varphi_{m,n}\|_{L^2(-1,1)}.
\end{multline}
Using the inequality
$$ \frac{2h^2}{\epsilon}AB\leq \frac{h^2}{8}A^2+\frac{8h^2}{\epsilon^2}B^2
$$
and the fact that $\epsilon\geq h^{1-\sigma}$, we can bound the left hand side of \eqref{eq.elliptic3} by $h^{2\sigma}$.
We now show~\eqref{eq.elliptic2} for $N= k\sigma$ by induction on $k$. The initial step is clear since the right hand side of \eqref{eq.elliptic2} is bounded by $h^{2\sigma}\leq h^{\sigma}$. Applying the induction assumption we get 
$$\|(1-\chi)(\frac x \epsilon) \varphi_{m,n}'\|_{L^2(-1,1)} \|(\chi' (\frac x  \epsilon) \varphi_{m,n}\|_{L^2(-1,1)} \leq C h^{k \sigma-1}
$$
which, from~\eqref{eq.elliptic3}  implies (using $\epsilon \geq h^{1- \sigma}$),  
$$\|(1-\chi)(\frac x  \epsilon) \varphi_{m,n}'\|_{L^2(-1,1)} + \|(1-\chi)(\frac x  \epsilon) \varphi_{m,n}\|_{L^2(-1,1)} \leq C h^{((k+1)\sigma}.
$$ 
\end{proof}

For $w_{h,\hbar}$, we expand it as
$$ w_{h,\hbar}=\sum_{|n|\sim \hbar^{-1},\lambda_{m,n}\sim h^{-1}}a_{m,n}\Phi_{m,n}.
$$
Note that for fixed $n$, applying Cauchy-Schwartz and Corollary \ref{Weyl}, we have
\begin{equation}\label{elliptic-2}
\begin{split}
 &\Big\|\big(1-\chi(\frac{x}{\epsilon})\big)\sum_{m:\lambda_{m,n}\sim h^{-1}}a_{m,n}\varphi_{m,n}(x)\Big\|_{L^2(-1,1)}+h\Big\|\big(1-\chi(\frac{x}{\epsilon})\big)\sum_{m:\lambda_{m,n}\sim h^{-1}}a_{m,n}\varphi_{m,n}'(x)\Big\|_{L^2(-1,1)}\\
 \leq &C_Nh^{N}\Big(\sum_{m:\lambda_{m,n}\sim h^{-1}}|a_{m,n}|^2 \Big)^{1/2} \left(\#\{m: \lambda_{m,n}\sim h^{-1} \}\right)^{1/2}\leq C_Nh^N h^{-1}\hbar^{-1/2}\leq C_N'h^{N-2}.
\end{split}
\end{equation}
Finally, using Plancherel's identity, we obtain that
$$ \|\big(1-\chi(\frac{x}{\epsilon})\big)h\partial_xw_{h,\hbar}\|_{L^2(\Omega)}+\|\big(1-\chi(\frac{x}{\epsilon})\big)w_{h,\hbar}\|_{L^2(\Omega)}\leq C_Nh^{N},
$$
for all $N\in\N$. This completes the proof of Proposition \ref{elliptique}.

\end{proof}

\subsection{Half-wave regime}

We shall use the following  elementary lemma.
\begin{lem}\label{timetruancation}
	For any $0<\epsilon_0<1$,  there exists a $C^1$ function $\varphi$, compactly supported in $[-1,1]$, such that
	\begin{enumerate}
		\item$ \varphi(t)\equiv 1,\forall |t|\leq 1-\epsilon_0, $ $0\leq \varphi(t)\leq 1$ and $ \varphi(t)\equiv 0,\forall |t|>1.
		$	
		\item $$
		\int_{1-\epsilon_0\leq |t|\leq 1}|\varphi'(t)|dt= 2.
		$$
	\end{enumerate}
	
\end{lem}
Introducing the Grushin energy norm
$$ \|f\|_{\dot{H}_G^1(\Omega)}^2:=\|\partial_xf\|_{L^2(\Omega)}^2+\|x\partial_yf\|_{L^2(\Omega)}^2.
$$
The following proposition is crucial in the half-wave regime.
\begin{prop}\label{Half-wavelemma1}
	Assume that $T>a_0$. Let $u_{h,\hbar,\delta}$ be the solution of \eqref{main-equation-bounded} with the localization property \eqref{localization-second}.
	Then there exist small constants $0<h_0,\hbar_0,\delta_0,\sigma_0<1 $,  and $C_{T,a_0}>0$ such that for all $0<h<h_0, 0<\delta<\delta_0, 0<\hbar<\hbar_0$, $2e_0<\rho\leq \sigma_0 \hbar^{-1/2}$ obeying  $\frac{e_0h^2}{2}<\hbar<\rho h^2$,  we have
	\begin{equation}\label{Ine:Half-wave-regime}
	\begin{split}
	\|u_{h,\hbar,\delta}(0)\|_{\dot{H}_{G}^1(\Omega)}^2\leq &C_{T,a_0}\int_{-T}^T\left(\|\phi(y)x\partial_yu_{h,\hbar,\delta}(t)\|_{L^2(\Omega)}^2+\|\phi(y)\partial_xu_{h,\hbar,\delta}(t)\|_{L^2(\Omega)}^2\right)dt\\
	+&C_{T,a_0}\int_{-T}^T\hbar^{-1}\|\phi(y)u_{h,\hbar,\delta}(t)\|_{L^2(\Omega)}^2dt+ C_{T,a_0}(1+\delta\hbar^{-1})\|u_{h,\hbar,\delta}(0)\|_{L^2(\Omega)}^2.
	\end{split}
	\end{equation}
 (recall that numerical constant $e_0$ satisfies $e_0>\frac{1}{2}$.)
\end{prop}
\begin{proof} We first give the proof in the simpler case where there is a single band in $\omega^c $, and by translation in $y$ we can assume that , $\omega^c=(-1,1)_x\times (-a,a))$, hence $\mathcal{L}(\omega)=2a$. In this case, the cut-off function $\phi$ defined in \eqref{phi} satisfies 
$$ \phi(y)=0 \text{ if } |y|\leq a;\quad \phi(y)=1 \text{ if } a_0<|y|\leq \pi.
$$	 
 To simplify the notation, we denote by $u=u_{h,\hbar,\delta}$. Pick $\varphi(t)$ as in the Lemma \ref{timetruancation}, with a small parameter $\epsilon_0>0$ to be chosen later. For $T>0$, we define $\varphi_T(t):=\varphi\left(\frac{t}{T}\right)$. Take $\chi_1(x)=\chi(\epsilon^{-1}x)$ for some function $\chi\in C_c^{\infty}(\R)$ with the property that $\chi(\zeta)\equiv 1 $ near $0$, where $\epsilon=\max\{ h^{1-\sigma},5c_1^{-1}\hbar h^{-1} \}$ as in Proposition \ref{elliptique}.

  Take $\chi_2\in C_c^{\infty}((-\pi,\pi))$ equal to $1$ on $(-a,a)$ and such that supp$(\chi_2')\subset $ supp$(\phi)$.  We define another cut-off $\phi_1(y)$ such that
	\begin{equation}\label{cutoff-intermediant}
	\mathrm{supp}(\chi_2')\subset \mathrm{supp}(\phi_1)\subset \mathrm{supp}(\phi),\quad \phi_1|_{\mathrm{supp}(\chi'_2)}\equiv 1,\quad \phi|_{\mathrm{supp}(\phi_1)}\equiv 1.
	\end{equation}
Moreover, we require that $	\mathrm{supp}(1-\chi_2)\subset\mathrm{supp}(\phi_1)$.
The proof is an application of positive commutator method, based on the simple commutator relation
\begin{equation}\label{commut}
 [\Delta_G,x\partial_x+2y\partial_y]=2\Delta_G.
\end{equation}
More precisely, with cutoffs $\varphi_T(t),\chi_1(x),\chi_2(y)$, we have
\begin{equation}\label{commutator-halfwave}
\begin{split}
&[i\partial_t+\Delta_G,\varphi_T(t)\chi_1(x)\chi_2(y)(x\partial_x+2y\partial_y)]
\\=&2\varphi_T(t)\chi_1(x)\chi_2(y)\Delta_G+i\varphi_T'(t)\chi_1(x)\chi_2(y)(x\partial_x+2y\partial_y)\\
+&\varphi_T(t)\left[\chi_2(y)(\chi_1''(x)+2\chi_1'(x)\partial_x)+x^2\chi_1(x)(\chi_2''(y)+2\chi_2'(y)\partial_y ) \right](x\partial_x+2y\partial_y).
\end{split}
\end{equation}
By developing the commutator and using the equation, we have $$([i\partial_t+\Delta_G,\varphi_T(t)\chi_1\chi_2(x\partial_x+2y\partial_y)]u,u )_{L^2(\R\times\Omega)}=0.$$
On the other hand, using \eqref{commutator-halfwave}, we obtain that
\begin{equation}\label{HFPf-1}
\begin{split}
-2\left(\varphi_T(t)\chi_1\chi_2\Delta_Gu,u \right)_{L^2(\R\times\Omega)}=& \underbrace{i\left(\varphi_T'(t)\chi_1\chi_2(x\partial_x+2y\partial_y)u,u \right)_{L^2(\R\times\Omega)}}_{\mathrm{I}}\\
+&\underbrace{\left(\varphi_T(t)\chi_2(\chi_1''+2\chi_1'\partial_x)(x\partial_x+2y\partial_y)u,u \right)_{L^2(\R\times\Omega)}}_{\mathrm{II}}\\
+&\underbrace{ \left(\varphi_T(t)x^2\chi_1(\chi_2''+2\chi_2'\partial_y  )(x\partial_x+2y\partial_y )u,u \right)_{L^2(\R\times\Omega)} }_{\mathrm{III}}.
\end{split}
\end{equation}
First, we observe that on the support of $1-\chi_1(x), \chi_1'(x), \chi_1''(x)$, we have $|x|>\epsilon$. Therefore, from Proposition \ref{elliptique}, together with the fact that $\|\partial_yu\|_{L^2(\Omega)}\sim \hbar^{-1}\|u\|_{L^2(\Omega)}\leq Ch^{-2}\|u\|_{L^2(\Omega)}^2$,
$$ |\mathrm{II}|\leq \mathcal{R}_1:=C_{T,N}h^N\|u(0)\|_{L^2(\Omega)}^2.
$$
Next, we observe that the supports of  $1-\chi_2(y),\chi_2'(y),\chi_2''(y)$ are both contained in $\{y:\phi_1(y)>0 \}$. Thus from integration by part, we obtain that
$$ |\mathrm{III}|\leq \mathcal{R}_2:=C\int_{\R}\int_{\Omega}\varphi_T(t)\phi_1(y)^2\left[|\partial_xu|^2+|x\partial_yu|^2+|u|^2 \right]dxdydt.
$$
Similarly, doing integration by part, we deduce that the left hand side of \eqref{HFPf-1} equals to
$$ 2\int_{\R}\varphi_T(t) \chi_1 (x) \chi_2 (y) \bigl( |\partial_x u|^2 + |x\partial_y u|^2\bigr) dydxdt + O(\mathcal{R}_1)+O(\mathcal{R}_2).
$$
Using again Proposition~\ref{elliptique} we can eliminate the $\chi_1$ cut off and get (adding another $\mathcal{R}_1$ error)
$$ 2\int_{\R}\varphi_T(t) \chi_2 (y) \bigl( |\partial_x u|^2 + |x\partial_y u|^2\bigr) dydxdt + O(\mathcal{R}_1)+O(\mathcal{R}_2).
$$
Moreover, 
$$ |\left(\varphi_T'(t)\chi_1\chi_2x\partial_xu,u\right)_{L^2(\R\times\Omega)}|\leq C_{T}\epsilon\|u(0)\|_{\dot{H}_G^1(\Omega)}\|u(0)\|_{L^2(\Omega)},
$$
since on the support of $\chi_1(x)$, $|x|\leq C\epsilon$. This term is in lower order and can be bounded by
$$\mathcal{R}_3\leq \epsilon\|u(0)\|_{\dot{H}_G^1(\Omega)}^2+C_{T,a'}\epsilon\|u(0)\|_{L^2(\Omega)}^2.
$$
 Thus, using that $\chi_2 + \phi \geq 1$ on $(- \pi, \pi)$, (because we can add $\mathcal{R}_2$ on the l.h.s.) we obtain from \eqref{HFPf-1} that 
\begin{equation}\label{HFPf-2}
\begin{split}
2\int_{\R}\varphi_T(t)\|u(t)\|_{\dot{H}_G^1(\Omega)}^2dt\leq &2|\left(\varphi_T'(t)\chi_2(y)y\partial_yu,u\right)_{L^2(\R\times\Omega)}|+ \mathcal{R}_1+\mathcal{R}_2+\mathcal{R}_3.
\end{split}
\end{equation}
Pick $a'\in(a_0,T)$, such that $\phi_1(y)>C_{a'}^{-1}>0$ for all $a'\leq |y|\leq \pi$. Therefore, the first term on the right side of \eqref{HFPf-2} can be majorized by
\begin{equation}\label{HFPf-3}
\begin{split}
&\underbrace{2a'\int_{(1-\epsilon_0)T\leq |t|\leq T} |\varphi_T'(t)|\|\partial_yu(t)\|_{L^2(|y|\leq a')}\|u(t)\|_{L^2(|y|\leq a')}dt}_{\mathrm{IV}}\\
+& \underbrace{C_{a'}^2\int_{\R}|\varphi_T'(t)|\|\phi_1(y)\partial_yu(t)\|_{L^2(\Omega)}\|\phi_1(y)u(t)\|_{L^2(\Omega)}dt}_{\mathrm{V}}.
\end{split}
\end{equation}
 Note that by our choice of $\phi_1$ and the localization property of $u=\psi(\hbar D_y)u$, we can write
$$ \phi_1(y)\partial_yu=\phi_1(y)\partial_y(\phi(y)u)=\phi_1(y)\partial_y\psi(\hbar D_y)(\phi(y)u)+\phi_1(y)\partial_y[\phi(y),\psi(\hbar D_y)]u.
$$
Using the fact that $\|\hbar\partial_y[\phi,\psi(\hbar D_y) ] \|_{L^2\rightarrow L^2}=O(\hbar)$,
 we can majorize the second term of \eqref{HFPf-3} by
$$ \mathrm{V}\leq \mathcal{R}_4:=C_{T,a'}\int_{-T}^T\hbar^{-1}\|\phi(y)u(t)\|_{L^2(\Omega)}^2dt+C_{T,a'}\|u(0)\|_{L^2(\Omega)}^2.
$$
It remains to treat IV.

	Using the spectral localization $||\hbar D_y|-1|\leq \delta$ of $u$, we can majorize IV by
\begin{equation}\label{c3}
\frac{2a'(1+3\delta)}{\hbar}\int_{(1-\epsilon_0)\leq |t|\leq T}|\varphi_T'(t)|\|u(t)\|_{L^2(\Omega)}^2dt.
\end{equation}	
From the conservation of mass and Lemma \ref{timetruancation}, we have that
$$ \eqref{c3} \leq \frac{2a'(1+3\delta_0)}{\hbar T(1-\epsilon_0)}\int_{\R}\varphi_T(t)\|u(t)\|_{L^2(\Omega)}^2dt.
$$ 
Denote by $\Pi_{\pm}$ be the projector to positive(negative) Fourier modes in $y$. We write
$$ \|u(t)\|_{L^2(\Omega)}^2=(\Pi_{+}u(t),\Pi_{+}u(t))_{L^2(\Omega)}+(\Pi_-u(t),\Pi_-u(t))_{L^2(\Omega)}.
$$
Hence we have 
$$ |((1\mp\hbar D_y)\Pi_{\pm}u(t),\Pi_{\pm}u(t))_{L^2(\Omega)}|\leq \delta \|u(t)\|_{L^2(\Omega)}^2.
$$
Now using the fact that $[\partial_x,x\partial_y]=\partial_y$ and $[\partial_x,\Pi_{\pm}]=0$, we obatin that
\begin{equation*}
\begin{split}
\|u(t)\|_{L^2(\Omega)}^2\leq &2\delta\|u(t)\|_{L^2(\Omega)}^2+2\hbar \|\partial_x\Pi_{+}u(t)\|_{L^2(\Omega)}\|x\partial_y\Pi_+u(t)\|_{L^2(\Omega)}\\+&2\hbar\|\partial_x\Pi_{-}u(t)\|_{L^2(\Omega)}\|x\partial_y\Pi_-u(t)\|_{L^2(\Omega)}\\
\leq &2\delta\|u(t)\|_{L^2(\Omega)}^2+\ \hbar\|u(t)\|_{\dot{H}_G^1(\Omega)}^2.
\end{split}
\end{equation*}
	Plugging into \eqref{HFPf-2}, we have
\begin{equation}\label{c4}
\begin{split}
2\|\varphi_T^{1/2}u\|_{L^2(\R;\dot{H}_G^1(\Omega))}^2\leq  & \frac{2a'(1+3\delta)}{T(1-\epsilon_0)}\|\varphi_T^{1/2}u\|_{L^2(\R;\dot{H}_G^1(\Omega))}^2+C_{T,a'}\hbar^{-1}\delta\|u(0)\|_{L^2(\Omega)}^2+\sum_{k=1}^4\mathcal{R}_k.
\end{split}
\end{equation}
Note that if $T>a_0$, we can choose $\delta_0>0, \epsilon_0>0$ sufficiently small, such that
$$  \frac{a'(1+3\delta)}{T(1-\epsilon_0)}<1.
$$
Then we choose $\hbar_0,h_0>0,\sigma_0>0$ sufficiently small, then for any $0<h<h_0, 0<\hbar<\hbar_0$,  $2e_0<\rho<\sigma_0\hbar^{-1/2}$, and $\frac{e_0 h^2}{2}<\hbar<\rho h^2$.
Substituting the expression of $\mathcal{R}_k$, the proof of Proposition \ref{Half-wavelemma1} is then complete in the particular case where there is only one band in $\omega^c$.
	
	 In the general case, there may be several bands (all with width at most $\mathcal{L}(\omega)$). 
In the previous proof, we replace the cut-off $\chi_2$ by a family of cut-off $\chi_2^j(y)$ equal to $1$ on $(a_j, b_j)$ and satisfying all the other assumptions of $\chi_2$ and we can perform the same estimates with the commutator relation~\eqref{commut} replaced by 
\begin{equation}\label{commuttbis}
 [\Delta_G,x\partial_x+2\bigl(y - \frac{ b_j + a_j} 2 \Bigr) \partial_y]=\Delta_G,
\end{equation}
(i.e. centering our estimates on the middle of the segment $(a_j, b_j))$
leading to the following version of~\eqref{HFPf-2}
\begin{equation}\label{HFPf-4}
2\int_{\R}\varphi_T(t)\|u(t)\|_{\dot{H}_G^1(\Omega)}^2dt\leq 2\sum_{j=1}^N \big|\left(\varphi_T'(t)\chi^j_2(y)\Bigl( y- \frac{ b_j + a_j} 2 \Bigr) \partial_yu,u\right)_{L^2(\R\times\Omega)}\big|+\sum_{k=1}^4\mathcal{R}_k,
\end{equation}
and we conclude the proof as previously.
\end{proof}
Next, we use a simple integration by part argument to pass to the $L^2$ observability estimate. 
\begin{cor}\label{HWCor1}
Under the same assumption as Proposition \ref{Half-wavelemma1}, we have
	$$ \|u_{h,\hbar,\delta}(0)\|_{L^2(\Omega)}^2\leq C_{T,a_0,b_0}\int_{-T}^T\|\phi(y)u_{h,\hbar,\delta}(t)\|_{L^2(\Omega)}^2.
	$$
\end{cor}
\begin{proof}
	The right side of \eqref{Ine:Half-wave-regime} is
	$$ \|u(0)\|_{\dot{H}_G^1(\Omega)}^2\sim h^{-2}\|u(0)\|_{L^2(\Omega)}^2.
	$$
	Since $\hbar^{-1}<\frac{2}{e_0h^2}$, the left side of \eqref{Ine:Half-wave-regime} can be majorized be
$$ C_{T,a_0}\int_{-T}^T\int_{\Omega}\phi(y)^2\left[|\partial_xu|^2+|x\partial_yu|^2\right]dxdydt+C'_{T,a_0}\left(1+\delta h^{-2}+h^{-2}\int_{-T}^T\int_{\Omega}\phi(y)^2|u|^2dtdxdy\right).
$$
Doing integration by part, we have
\begin{multline*}
\int_{-T}^T\int_{\Omega}\phi(y)^2\left[|\partial_xu|^2+|x\partial_yu|^2\right]dxdydt\\
=-\int_{-T}^T\int_{\Omega}\left[\phi(y)^2\Delta_Gu\ov{u}dxdydt+2\phi\phi'(y)x^2\partial_yu\ov{u}\right]dxdydt.
\end{multline*}
By Cauchy-Schwartz, the right hand side can be bounded by
$$ \|\phi(y)u\|_{L^2((-T,T)\times \Omega)}\|\Delta_G u\|_{L^2((-T,T)\times\Omega)}+\|\phi(y)u\|_{L^2((-T,T)\times \Omega)}\|u\|_{\dot{H}_G^1(\Omega)}.
$$
From the localization of the spectral, we have
$$\int_{-T}^T\int_{\Omega}\phi(y)^2\left[|\partial_xu|^2+|x\partial_yu|^2\right]dxdydt\leq \frac{C_T}{h^2}\|\phi(y)u\|_{L^2((-T,T)\times\Omega)}\|u(0)\|_{L^2(\Omega)}. 
$$
Finally, applying Young's inequality of the form
$$ ab\leq \epsilon a^2+C(\epsilon)b^2,
$$
the proof of Corollary \ref{HWCor1} is complete.
\end{proof}

The following proposition shows that it is the regime for which $\hbar$ is close to $h^2$ that induces the condition $T>a_0$.
\begin{prop}\label{pseudo-HWregime}
Assume that $T>0$, then there exist $b_0>0, \rho_0>0, h_0>0, \hbar_0>0,\delta_0>0$, and $C_T>0$, such that for all $0<h<h_0, 0<\hbar<\hbar_0, 0<\delta<\delta_0$, and $\rho\geq \rho_0$, obeying $\rho h^2<\hbar\leq b_0h$, we have
\begin{equation}\label{Ine:pseudo-HWregime}
\begin{split}
\|u_{h,\hbar,\delta}(0)\|_{L^2(\Omega)}^2\leq C_T\int_{-T}^T\|\phi(y)u_{h,\hbar,\delta}(t)\|_{L^2(\Omega)}^2dt.
\end{split}
\end{equation}
\end{prop}
\begin{proof}
We use the same positive commutator method as in the proof of Proposition \ref{Half-wavelemma1}, but the proof is much simpler in the present situation. We adapt the notation there. Indeed, from \eqref{HFPf-1}, the term $|\mathrm{II}|,|\mathrm{III}|$ can be majorized by the same bound $\mathcal{R}_1,\mathcal{R}_2$. For I=$i(\varphi_T'(t)\chi_1\chi_2(x\partial_x+2y\partial_y)u,u)_{L^2(\R\times\Omega)}$, we estimate it as 
$$ |(\varphi_T'(t)\chi_1\chi_2x\partial_xu,u )_{L^2(\R\times\Omega)}|\leq C_T\epsilon\|u(0)\|_{\dot{H}_G^1(\Omega)}\|u(0)\|_{L^2(\Omega)},
$$
where $\epsilon=\max\{h^{1-\sigma}, 5c_1^{-1}\hbar h^{-1} \}$,
and
$$ |(\varphi_T'(t)\chi_1\chi_2y\partial_yu,u )_{L^2(\R\times\Omega)}|\leq C_T\|\partial_yu(0)\|_{L^2(\Omega)}\|u(0)\|_{L^2(\Omega)}\leq C_T\hbar^{-1}\|u(0)\|_{L^2(\Omega)}^2.
$$
Therefore,
\begin{equation}
\begin{split}
\int_{\R}\varphi_T(t)\|u(t)\|_{\dot{H}_G^1(\Omega)}^2dt\leq & C_T\int_{-T}^T\int_{\Omega}\phi(y)^2\left[|\partial_xu|^2+|x\partial_yu|^2 \right]dxdydt\\
+&C_T\epsilon\|u(0)\|_{\dot{H}^1(\Omega)}\|u(0)\|_{L^2(\Omega)}+C_T\hbar^{-1}\|u(0)\|_{L^2(\Omega)}^2.
\end{split}
\end{equation}
Note that $\|u(t)\|_{\dot{H}_G^1(\Omega)}=\|u(0)\|_{\dot{H}_G^1(\Omega)}\sim h^{-2}$, after doing integration by part as in the proof of Corollary \ref{HWCor1}, we obtain that
$$ \|u(0)\|_{L^2(\Omega)}^2\leq C_T\int_{-T}^T\|\phi(y)u(t)\|_{L^2(\Omega)}^2dt+C_T\epsilon\|u(0)\|_{L^2(\Omega)}^2+C_Th^2\hbar^{-1}\|u(0)\|_{L^2(\Omega)}^2.
$$ 
Note that $\hbar h^{-1}\leq b_0,\quad   h^2\hbar^{-1}\leq \rho^{-1}$. We choose $\rho_0>0$ large enough and $b_0>0$ small enough, such that
$$ C_Tb_0+C_T\rho_0^{-1}<\frac{1}{2}.
$$
This completes the proof of Proposition \ref{pseudo-HWregime}.
\end{proof}

\subsection{\texorpdfstring{$L^2$}{L2} observability in the half-wave regime}

\begin{proof}[Proof of Proposition \ref{regime:Half-wave}]
	Fix $R>1$, close enough to $1$ such that 
	$$ \max\{R-1, 1-R^{-1}\}<\delta,
	$$
	for some $\delta$ fixed as in the Corollary \ref{HWCor1} and Proposition \ref{pseudo-HWregime}. Choose a cut-off funtion $\beta_R\in C_c^{\infty}((R^{-1},R))$, $\beta_R\geq 0$ and
	$$ \beta_0(s)+\sum_{k=1}^{\infty}\beta_{R,k}(s)=1,\quad \forall s>0.
	$$ 
	where $\beta_{R,k}(s):=\beta_R(R^{-k}s)$ and $\beta_0(s)\in C_c^{\infty}\left(\left[0,\frac{R+1}{2R}\right)\right)$.

	We will apply Proposition \ref{HWCor1} for the semi-classical parameter $\hbar=\lambda=R^{-k}$ with large $k$, and denote by $\rho=\rho_0$ as in Proposition \ref{pseudo-HWregime}. 	Fix the small parameters $\sigma_0>0$ as in Proposition \ref{Half-wavelemma1}, and we choose $h_0>0$ smaller than the one required in Proposition \ref{Half-wavelemma1}, Proposition \ref{pseudo-HWregime}, and $\sigma_0\rho^{-3/2}$. Note that in this case, if
	$ \frac{e_0h^2}{2}<\hbar<\rho h^2,
	$
	the constraint $\rho<\sigma_0 \hbar^{-1/2}$ to apply Proposition \ref{Half-wavelemma1} is satisfied automatically.
	 Denote simply by $u_h=\Pi_{h}^{b_0h}u$, we know that $\beta_{R,k}(|D_y|)u_h=0$ unless 
	$$ N_h:= -\frac{\log(b_0h)}{\log R}\leq k\leq M_h:=-\frac{\log (e_0h^2/2)}{\log R}.
	$$
	Now from almost orthogonality
	$$ C_R^{-1}\|u_h\|_{L^2(\Omega)}^2\leq  \sum_{k\geq N_h}\|\beta_{R,k}(|D_y|)u_h\|_{L^2(\Omega)}^2\leq C_R\|u_h\|_{L^2(\Omega)}^2
	$$
	and Corollary \ref{HWCor1}, Proposition \ref{pseudo-HWregime}, we have that 
	\begin{equation}\label{Littlewood-sum}
	\begin{split}
	\|u_h(0)\|_{L^2(\Omega)}^2\leq &C_R\sum_{ N_h\leq k\leq M_h}\left\|\beta_{R,k}(|D_y|)u_h(0)\right\|_{L^2(\Omega)}^2\\ \leq &C_{T,R}\sum_{N_h\leq k\leq M_h}\int_{-T}^T\|\phi(y)\beta_{R,k}(|D_y|)u_h(t)\|_{L^2(\Omega)}^2dt.
	\end{split}
	\end{equation}
	Note that $\|[\phi(y),\beta_{R,k}(|D_y|)]\|_{L^2(\T)\rightarrow\ L^2(\T)}\leq C_R R^{-k}$,  we have
	$$ \eqref{Littlewood-sum} \leq C_{T,R}\int_{-T}^T\|\Pi_h^{b_0h}(\phi(y)u_h(t))\|_{L^2(\Omega)}^2dt+C_{T,R}\sum_{N_h\leq k\leq M_h}R^{-2k}\int_{-T}^T\|u_h(t)\|_{L^2(\Omega)}^2dt.
	$$
	The second term on the right hand side can be bounded from above by
	$$ C_{T,R}R^{-2N_h}\|u_h(0)\|_{L^2(\Omega)}^2\leq C_{T,R}(b_0h)^2\|u_h(0)\|_{L^2(\Omega)}^2.
	$$
	By taking $h<h_0$ small enough, it can be absorbed to the left side. The proof of Proposition \ref{regime:Half-wave} is now complete.
\end{proof}

\section{Dispersive regime I: \texorpdfstring{$b_0 h^{-1} \leq |D_y| \leq b_0^{-1 } h^{-1}$:}{} semi-classical control}\label{sec.5}
Fix $0<b_0<1$ from the last section. 
Denote by
$$  \Pi_{h,b_0}=\psi(h^2\Delta_G)\chi_0(b_0hD_y)(1-\chi_0(b_0^{-1}hD_y)).
$$
We will prove the following proposition in this section.
\begin{prop}\label{Regime:GCC}
	Given $T>0$, there exsit $h_0>0$, such that for all $0<h<h_0$, the following inequality holds true for all $u_0\in L^2(\Omega)$:
	\begin{equation}\label{eq:regimeGCC}
	\|\Pi_{h,b_0}u_0\|_{L^2(\Omega)}^2\leq C_T\int_{-2T}^{2T}\|\phi(y)e^{it\Delta_G}\Pi_{h,b_0}u_0\|_{L^2(\Omega)}^2.
	\end{equation}	
\end{prop}

\begin{proof}
The proof follows essentially the strategy of Lebeau (\cite{Le}).	Making the change of variable $t=hs$,  $v(s,x,y)=e^{it\Delta_G}\Pi_{h,b_0}u_0(x,y)$ and $v_0=\Pi_{h,b_0}u_0$, the function $v$ satisfies the semi-classical Schr\"odinger equation
	\begin{equation}\label{semi-Schrodinger}
	\begin{split}
	ih\partial_sv+h^2\Delta_Gv=0.
	\end{split}
	\end{equation} 
	By using the conservation of $L^2$ norm, the proof of \eqref{Regime:GCC} reduces to the semi-classical observability: there exists $T_1=T_1(b_0)>0$, such that
	\begin{equation}\label{ob:semi}
	\|v_0\|_{L^2(\Omega)}^2\leq C_{T_1}\int_{-T_1}^{T_1}\|\phi(y)v(s,\cdot)\|_{L^2(\Omega)}^2ds.
	\end{equation}
	Indeed, to see how \eqref{ob:semi} implies \eqref{eq:regimeGCC}, we change back to $t$ variable and rewrite \eqref{ob:semi} as
	$$ \|\Pi_{h,b_0}u(0,\cdot)\|_{L^2(\Omega)}^2\leq \frac{C_{T_1}}{h}\int_{-hT_1}^{hT_1}\|\phi(y)\Pi_{h,b_0}u(t,\cdot)\|_{L^2(\Omega)}^2dt.
	$$ 
	For any $k\in\Z$, applying the inequality above to the initial data $\Pi_{h,b_0}u_0(khT_1,\cdot)$ and using the $L^2$ conservation, we obtain that
	$$ \|\Pi_{h,b_0}u_0\|_{L^2(\Omega)}^2=\|\Pi_{h,b_0}u(khT_1,\cdot)\|_{L^2(\Omega)}^2\leq \frac{C_{T_1}}{h}\int_{(k-1)hT_1}^{(k+1)hT_1}\|\phi(y)\Pi_{h,b_0}u(t,\cdot)\|_{L^2(\Omega)}^2dt.
	$$
	Summing over $ -\big\lfloor\frac{T}{T_1h}\big\rfloor+1 \leq k\leq \big\lfloor\frac{T}{T_1h}\big\rfloor-1$, we obtain \eqref{eq:regimeGCC}.
	
	Now we prove \eqref{ob:semi}.	With a little abuse of the notation, we denote by $v=\Pi_{h,b_0}v$.
	Again by conservation of the $L^2$ norm, it is sufficient to show that 
	\begin{equation}\label{ob:semi-1}
	\|v\|_{L^2((-2T_1,0)\times\Omega)}^2\leq C_{T_1}\|\phi(y)v\|_{L^2((-T_1,T_1)\times\Omega)}^2.
	\end{equation}
	Denote by $X=\R_t\times \Omega$ with boundary $\partial X=\R_s\times\partial\Omega$. 
	Take $\chi\in C_c^{\infty}(\R)$ such that $\chi|_{\textrm{ supp}(\psi)}\equiv 1$. Then, using that if $u= e^{ish\Delta_G} u_0$, $(1- \chi) (hD_s) u = (1- \chi)(- h^2 \Delta_G) u$, we get 
	$$ (1- \chi)(hD_s) e^{ish\Delta_G} \psi(h^2\Delta_G) =0.
	$$
%
	Therefore, it suffices to prove \eqref{ob:semi-1} for time-frequency localized functions $\chi(hD_s)v$. We argue by contradiction. If \eqref{ob:semi-1} is not true, then we obtain  sequences $h_k\rightarrow 0$ and $v_k=\chi(h_kD_s)\Pi_{h,b_0}v_k, $ satisfying
	\begin{equation}\label{contradiction} (ih_k\partial_s+h_k^2\Delta_G)v_k=0,\quad \|v_k(0)\|_{L^2(\Omega)}=1,\quad \textrm{ and }\quad  \|\phi(y)v_k\|_{L^2((-T_1,T_1)\times \Omega)}\rightarrow 0 
	\end{equation}
	as $k\rightarrow\infty$.
	From \cite{Ge91}\footnote{here we adopt the notation from \cite{B04}}, after substracting a subsequence, still denoted by $(v_k)$, there exists a Radon measure on $Z=j(\mathrm{Ch}(P))$, such that for any sum of the interior and tangential operators $A_h=A_{i,h}+A_{\partial,h}$, compactly supported in $s$, we have
	$$ \lim_{k\rightarrow\infty}(A_{h_k}v_k, v_k )_{L^2(X)}=\langle \mu, \kappa (\sigma(A_h))\rangle.
	$$
	Away from the boundary $\partial X=\R\times \partial\Omega$, this measure is invariant along the characteristic flow of $H_p=\partial_s-2\xi\partial_x+2x\eta^2\partial_{\xi}-2x^2\eta\partial_y$, where $p=-\sigma-\xi^2-x^2\eta^2$ is the principal symbol. Near the boundary, the cotangent bundle can be decomposed as
	$ T^*\partial X=\widetilde{\mathcal{E}}\cup \widetilde{\mathcal{H}}\cup \widetilde{\mathcal{G}}$, where
	\begin{equation*}
	\begin{split}
	&\widetilde{\mathcal{E}}=\{(x,s,y;\sigma,\eta): -\sigma-x^2\eta^2<0 \},\quad \widetilde{\mathcal{H}}=\{(x,t,y;\tau,\eta): -\sigma-x^2\eta^2>0 \},\\
	&\widetilde{\mathcal{G}}:=\{(x,s,y;\sigma,\eta): \sigma=-x^2\eta^2 \}.
	\end{split}
	\end{equation*} 
	Note that for any point $\rho\in\widetilde{\mathcal{G}}$, $(H_p^2x)|_{\rho}>0$ if $x(\rho)=-1$ and $(H_p^2x)|_{\rho}<0$ if $x(\rho)=-1$ where $\rho\mapsto x(\rho)$ is the boundary defining function. Therefore $\widetilde{\mathcal{G}}=\widetilde{\mathcal{G}^{2,+}}$, namely it consists only the diffractive points. It follows from \cite{B04}  that $\mu(\widetilde{\mathcal{H}})=\mu(\widetilde{\mathcal{E}})=\mu(\widetilde{\mathcal{G}^{2,+}})=0.$ Moreover, $\mu$ is invariant along the generalized bicharacteristic flow. Note that away from the reflexion points at the boundary $\{ x= \pm 1\}$, the flow is given by
	$$ \dot{s}=-\sigma= \tau_0,\quad \dot{y}=2x^2\eta,\quad \dot{\eta}=0 \Rightarrow \tau = \tau_0, \dot{x} =2\xi, \dot{\xi} = -2x \eta_0^2
	$$ 
	with initial data $-\sigma_0\in\textrm{supp}(\psi)\subset (\frac 1 4 , 4)$, $b_0\leq|\eta_0|\leq b_0^{-1}$. Integrating the flow,  in the $x,\xi$ variables gives an ellipse away from the boundary $\{ x= \pm1\}$. As a consequence, the variable $x$ is bounded away from $0$ for long sequences of time, and hence $y$ is strictly increasing for these sequences of time (and increasing otherwise). This shows  that there exists $T_1=T_1(b_0)>0$, such that for any $\rho$ on phase space,  $y(\pm T_1,\rho)\in \textrm{supp}(\phi)$, the controlled region. As in the assumption \eqref{contradiction}, $\mu|_{(-T_1,T_1)\times\textrm{supp}(\phi) }=0$ hence $\mu|_{(-2T_1,0)\times\Omega}=0$. This contradicts to $\|v_k\|_{L^2((-2T_1,0)\times \Omega)}=2T_1$, since $\|v_k(0)\|_{L^2(\Omega)}=1$ and $t\mapsto \|v_k(t)\|_{L^2(\Omega)}$ is a constant function. The proof of Proposition \ref{Regime:GCC} is now complete.
	
\end{proof}


\section{Dispersive regime II: \texorpdfstring{$ h^{-\epsilon}\leq |D_y|\leq b_0 h^{-1}$}{}}
We fix $0<b_0<1$ from the last sections. We fix another cutoff $\psi_0\in C_c^{\infty}(\R;[0,1])$ which is identically $1$ near the origin.  Denote by
$$ \underline{\Pi}_{h,b_0}^{\epsilon}:=\psi\left(h^2\Delta_G\right)\chi_0(b_0^{-1}hD_y)(1-\psi_0(h^{\epsilon}D_y) )
$$

In this section, we prove the following proposition:

\begin{prop}\label{ob:Regime2}
	Given $T>0$, there exist $h_0>0$, $C_T>0$,  such that the following observability holds true for all $0<h<h_0$ and $\epsilon>0$:
	$$ \|\underline{\Pi}_{h,b_0}^{\epsilon}u_0\|_{L^2(\Omega)}^2\leq C_T\int_0^T\|\phi(y)e^{it\Delta_G}\underline{\Pi}_{h,b_0}^{\epsilon_0}u_0\|_{L^2(\Omega)}^2dt + C_Th^{2\epsilon}\|u_0\|_{L^2(\Omega)}^2.
	$$

\end{prop}

The proof of this theorem will be decomposed into several lemmas. 

\subsection{Control of vertical propagation}
We will use the positive commutator method to control the vertical propagation. The goal is to control the norm $\|\partial_yu\|_{L^2}$. In the half-wave regime, this norm can be controlled by $\|x\partial_yu\|_{L^2}$and $\|\partial_xu\|_{L^2}$ by spending some explicit classical time, thanks to the hypoellipticity. Away from half-wave regime, we can control $\|\partial_yu\|_{L^2}$ directly by $\|x\partial_yu\|_{L^2}$
, exploiting the horizontal propagation in semi-classical time, which is small in the classical time scale. 
\begin{lem}[Horizontal propagation]\label{propagation}
	Given $0<b_0<\frac{1}{10}$ and $0<r_0<1$, there exists $h_0>0$, such that the following is true for all $0<h<h_0$:  If $(w_h)_{h>0}$ is a family of solutions of semi-classical equations
	$$ ih\partial_sw_h+h^2\Delta_Gw_h=0, \quad w_h|_{\partial\Omega}=0
	$$
	with spectral-localized property:
	$$ w_h=\psi_1\left(-h^2\Delta_G\right)\chi_0(b_0^{-1}hD_y)w_h,
	$$
	where $\psi_1\in C_c^{\infty}(\R)$ and $\chi_0\in C_c^{\infty}(\R)$ are two cutoff functions with supports $[-2,-\frac{1}{2}]\cup [\frac{1}{2},2]$ and $[-2,2]$, with respectively. 
	Then there exists a uniform constants $\kappa>0$,  independent of $b_0$ and $h$   such that for all $T_0>\kappa r_0$, we have
	\begin{equation}\label{propagation-semi}
	\|w_h(0)\|_{L^2(\Omega)}^2\leq C_{T_0}\int_0^{T_0}\|w_h(t)\|_{L^2(r_0<|x|<1)}^2dt.
	\end{equation} 
\end{lem}
\begin{proof}
This is again  a consequence of propagation arguments. Indeed, as in Section~\ref{sec.5} we can insert a cutoff $\chi(- hD_s)$, $\chi=1$ on the support of $\psi$, $\chi$ supported in $(\frac 1 4 , 4)$ and arguing by contradiction and defining defect measures, we get a contradiction if we can ensure that 
	for all $(t_0 =0, -\sigma_0\in \text{supp}(\chi), x_0 \in [-1,1], y_0 \in \T^1, \xi_0, \eta_0 \in [ 0,b_0])$ such that $- \sigma_0 - \xi_0 ^2- x_0^2 \eta_0^2$, 
	the point on the { bicharacteristic} i.e. the solution of the system
	\begin{equation}\label{flow}
	\dot{s} = - \sigma _0, \dot{x} = 2 \xi, \dot{\xi} = -2x \eta_0, \dot{y} = 2x^2 \eta_0
	\end{equation}
	reaches the region 
	$\{ |x|\in (r_0, 1)\}$ at a time $s\in (0, T_0)$. Notice that 
	$$ \xi_0^2= - \sigma_0 - x_0^2 \eta_0 \in (- \frac 1 4 - \frac 1 {100}, 4+ \frac 1 {100}), $$ and consequently  the $(x,\xi)$ integration of~\eqref{flow} gives again an ellipse and the geometric assumption is satisfied.
	 This completes the proof of Lemma  \ref{propagation}.

\end{proof}
Denote by $u_h(t)=\underline{\Pi}_{h,b_0}^{\epsilon}e^{it\Delta_G}u_0$, we have the following 
\begin{lem}\label{vertical-propagation}
	Let $T>0$, there exists $C_T>0$, and $h_0>0$,  such that for all $0<h<h_0$, we have
	\begin{multline}\label{verticle}
	\|x\partial_yu_h(t)\|_{L^2((-T,T);L^2(\Omega))}^2+	\|u_{h}(0)\|_{H_y^1(\Omega)}^2\\
	\leq C_T\int_{-3T}^{3T}\|\phi(y)x\partial_yu_h(t)\|_{L^2(\Omega)}^2dt+C_T\|u_h(0)\|_{L^{2}(\Omega)}^2.
	\end{multline}
\end{lem}
\begin{proof}
From minor modification as in the proof of Proposition \ref{Half-wavelemma1}, it suffices to deal with the case where $\phi(y)=0$ on $(-a,a)$. Denote by
	$ \varphi_T(t):=\varphi(T^{-1}t)
	$ the time cutoff. 
	Take $\chi\in C_c^{\infty}(\mathbb{T}_y)$ such that $\chi|_{\mathbb{T}_y\setminus \omega}\equiv 1$ and supp$(\chi')\subset \textrm{ supp }(\phi)\subset \omega$. 
	By direct calculation, one verifies that
	\begin{equation}\label{commutator1*}
	\begin{split}
	[i\partial_t+\Delta_G,\varphi_T(t)\chi(y)y\partial_y]=&
	2\varphi_T(t)\chi(y)(x\partial_y)^2+i\varphi_T'(t)\chi(y)y\partial_y \\
	+&x^2(y\chi''(y)+2\chi'(y))\varphi_T(t)\partial_y+2\varphi_T(t)\chi'(y)y(x\partial_y)^2.
	\end{split}
	\end{equation}
	Taking the inner product of $[i\partial_t+\Delta_G, \varphi_T(t)\chi(y)y\partial_y]u_h$ and $u_h$, using the equation, we have
	\begin{equation}\label{1}
	\begin{split}
	&2\int_{\R\times M} \varphi_T(t)\|\chi(x)x\partial_y u_{h}(t,x,y)\|_{L^2(\Omega)}^2dt \\
	\leq & a\int_{\R\times M}|\varphi_T'(t)|\|u_{h}(t)\|_{L^2(\Omega)}\|\partial_yu_{h}(t)\|_{L^2(\Omega)}dt\\
	+&C\int_\R \varphi_T(t)\|\phi(y)x\partial_yu_{h}(t)\|_{L^2(\Omega)}^2+C\int_{\R}\varphi_T(t)\|u_{h}(t)\|_{L^2(\R\times \Omega)}^2.
	\end{split}
	\end{equation}
	Next, we change the time scale by setting 
	$ w_{h}(s,x,y):=\partial_yu_{h}(t,x,y), t=sh.
	$
	Note that $w_{h}$ satisfies the semi-classical equation
	$ (ih\partial_s+h^2\Delta_G)w_{h}=0.
	$
	Applying Lemma \ref{propagation}, we obtain that for any $0<r_0<1$, $T_0>\kappa r_0$,
	$$ \|w_{h}(0)\|_{L^2(\Omega)}^2\leq C_{T_0}\int_0^{T_0}\|w_{h}(s)\|_{L^2(r_0<|x|<1)}^2ds.
	$$
	Back to the time variable $t$ and the original function $u_{h}$, exploiting again the conservation of $L^2$ norm, we obtain that for any $T>0$, 
	$$ \|\partial_yu_{h}(0)\|_{L^2(\Omega)}^2\leq C_T\int_{-T}^T\|\partial_yu_{h}(t)\|_{L^2(r_0<|x|<1)}^2dt.
	$$
	Note that
	$$ \|\partial_yu_{h}(t)\|_{L^2(r_0<|x|<1)}^2\leq \frac{C_T}{r_0^2}\int_{\R}\varphi_T(t)\|x\partial_yu_{h}(t)\|_{L^2(\Omega)}^2dt.
	$$
	Plugging into \eqref{1}, and using the conservation of the norms $\|\partial_yu_{h}(t)\|_{L^2(\Omega)}^2, \|u_{h}(t)\|_{L^2(\Omega)}^2$ as well as the Young's inequality\footnote{We use $2AB\leq \epsilon A^2+\epsilon^{-1}B^2$. }, we have
	\begin{equation*}
	\begin{split}
	\|\partial_yu_{h}(0)\|_{L^2(\Omega)}^2\leq & C\int_{-2T}^{2T}\|\phi(y)x\partial_yu_{h}(t)\|_{L^2(\Omega)}^2+C_T\|u_{h}(0)\|_{L^2(\Omega)}^2.
	\end{split}
	\end{equation*}
	Replacing $T$ to $3T/2$ in the definition of $\varphi_T$, we obtain from \eqref{1} that:
	\begin{equation}\label{partial-energy}
	\begin{split}
	\|x\partial_yu_h(t)\|_{L^2((-T,T)\times \Omega)}^2\leq C\int_{-3T}^{3T}\|\phi(y)x\partial_yu_h(t)\|_{L^2(\Omega)}^2+C_T\|u_h(0)\|_{L^2(\Omega)}^2,
	\end{split}
	\end{equation}
	This completes the proof of Lemma \ref{vertical-propagation}.
\end{proof}

\subsection{\texorpdfstring{$L^2$}{L2} observability for the rapid propagation regime}
\begin{proof}[Proof of Proposition \ref{ob:Regime2}]
	For $u_h(t)=\underline{\Pi}_{h,b_0}^{\epsilon}e^{it\Delta_G}u_0$, without loss of generality, we may assume that  $$u_h=(1-\psi_0(h^{\epsilon}D_y)) )\chi_0(b_0^{-1}hD_y)u_h.$$We write  $$\phi(y)xh\partial_yu_h=xh\partial_y\chi_0((2b_0)^{-1}hD_y) (\phi(y)u_h )+O_{L^2(\Omega)}(h). $$
	Hence from Lemma \ref{vertical-propagation}, we obtain that
	\begin{equation}\label{6.2-1}
	\|h\partial_yu_h(0)\|_{L^2(\Omega)}^2\leq C_T\int_{-3T}^{3T}\|h\partial_y\phi(y)u_h(t)\|_{L^2(\Omega)}^2dt+C_Th^2\|u_h(0)\|_{L^2(\Omega)}^2.
	\end{equation}
	By assumption, $\mathcal{F}_yu_h(\cdot, k)=0$ for all $|k|\leq Ch^{-\epsilon}$, thus we have that
	\begin{equation*}
	\begin{split}
	2C_Th^2\|u_h(0)\|_{L^2(\Omega)}^2\leq &2C_T(Ch)^{2\epsilon}\|h\partial_yu_h(0)\|_{L^2(\Omega)}^2\\ \leq &C_Th^{2\epsilon}\int_{-3T}^{3T}\|h\partial_y\phi(y)u_h(t)\|_{L^2(\Omega)}^2dt+C_Th^{2+2\epsilon}\|u_h(0)\|_{L^2(\Omega)}^2.
	\end{split}
	\end{equation*}
	Absorbing the second term on the right hand side to the left, we obtain that
	\begin{align}\label{partialy}
\|\partial_yu_h(0)\|_{L^2(\Omega)}^2\leq C_T\int_{-3T}^{3T}\|\partial_y\phi(y)u_h(t)\|_{L^2(\Omega)}^2.
 \end{align}
	To get the $L^2$ estimate, we write the Littlewood-Paley decomposition 
	$$ u_h=\sum_{j: h^{-\epsilon}\leq 2^j\leq b_0h^{-1}}\psi(2^{-j}D_y)u_h.
	$$
	Note that we may replace $u_h$ by $\psi(2^{-j}D_y)u_h$ in \eqref{partialy} with $h^{-\epsilon}\leq 2^j\leq \frac{b_0}{h}$. From the commutator bound
	$$ \|[\phi(y),\psi(2^{-j}D_y)]\|_{L^2\rightarrow L^2}\leq C2^{-j},
	$$ 
	we have
	$$ \|\psi(2^{-j}D_y)u_h(0)\|_{L^2(\Omega)}^2\leq C_T\int_{-3T}^{3T}\|\psi(2^{-j}D_y)(\phi(y)u_h(t))\|_{L^2(\Omega)}^2+2^{-2j}C_T\|u_h(0)\|_{L^2(\Omega)}^2.
	$$
	Summing the inequality above for $j\in[\log_2(h^{-\epsilon}),\log_2(b_0h^{-1})]$, we complete the proof of Proposition \ref{ob:Regime2}, provided that $h<h_0$ is small enough. 
\end{proof}


\section{Non-semiclassical dispersive regime II: \texorpdfstring{$  |D_y|\leq h^{-\epsilon}$ {}}: bottom vertical frequencies}		

Recall that $\psi\in C_c^{\infty}(\R)$ is supported on $[-2,-\frac{1}{2}]\cup[\frac{1}{2},2] $ and  $\psi_0\in C_{c}^{\infty}(\R)$ is supported on $[-1,1]$. Let
$$ u_h=\psi(h^2\Delta_G)\psi_0(h^{\epsilon}D_y)u.
$$		
In this section, we prove the following result:
\begin{prop}\label{Observation-Lowfrequency}
	For any $T>0$, there exist $C_T>0$, $h_0>0$, such that for sufficiently small $\epsilon>0$ and all $0<h<h_0$, we have
	\begin{equation*}
	\|u_h(0,\cdot)\|_{L^2(\Omega)}^2\leq C_T\int_{-T}^{T}\|\phi(y)e^{it\Delta_G}u_h(0,\cdot)\|_{L^2(\Omega)}^2dt+C_Th^{1-4\epsilon}\|u(0,\cdot)\|_{L^2(\Omega)}^2.
	\end{equation*}
\end{prop}
Inspired by \cite{BZ4}, we use normal form method. The key point is to search for a microlocal transformation
$$ v=(1+hQD_y^2)u
$$
for some suitable semi-classical pseudo-differential operator $Q=q(x,hD_x)$, such that the conjugated equation (satisfed by $v$) is
$$ i\partial_tv+\partial_x^2v+M\partial_y^2v=\textrm{ errors },
$$  
where 
$$M = \frac 1 2 \int_{-1}^1 x^2dx$$
is the mean value of $x^2$. 
Then we will be able to use the following theorem:
\begin{thm}[\cite{Ja},\cite{BZ4},\cite{AM}]\label{ob:Schrodinger}
	Let $\Delta_M=\partial_x^2+M\partial_y^2$. Then for any non-empty open set $\omega_0\subset\T^2$ and $T>0$, the observability
	\begin{equation}\label{eq:obSchrodinger}
	\|f\|_{L^2(\T^2)}^2\leq C_T\int_{-T}^T\|e^{it\Delta_M}f\|_{L^2(\omega_0)}^2dt
	\end{equation}
	holds true for any $f\in L^2(\T^2)$.
\end{thm}

However,  dealing with Dirichlet boundary value problem induces difficulties and consequently,  we prefered to extend the analysis to the periodic setting.

\subsection{Periodic extension }
Let us introduce several notations. Denote by
$$ \widetilde{\mathbb{T}}:=[-1,3]\slash\{-1,3\} \textrm{ and } \widetilde{\mathbb{T}}^2:=\widetilde{\mathbb{T}}_x\times \mathbb{T}_y, \Omega_{*}=(-1,3)_x\times\mathbb{T}_y.
$$
Define
$$ a(x)=x, \textrm{ if }|x|\leq 1 \textrm{ and } a(x)=2-x, \textrm{ if } 1\leq x\leq 3,
$$
and the operator
$$ P_a:=\partial_x^2+a(x)^2\partial_y^2.
$$ 
Note that $a(x)$ and $a^2(x)$ are Lipschitz functions on $\widetilde{\mathbb{T}}$.
Denote by 
$$ H_a^k(\widetilde{\mathbb{T}}^2):=\{f\in \mathcal{D}'(\widetilde{\mathbb{T}}^2): P_a^jf\in L^2(\widetilde{\mathbb{T}}^2),\forall 0\leq j\leq k \}
$$
the associated function spaces and the domain of $P_a$ is $D(P_a)=H_{a}^2(\widetilde{\mathbb{T}}^2).$ 
Note that $D(\Delta_G)=H_{G,0}^1(\Omega)\cap H_G^2(\Omega)$.
Consider the extension map:
$$ \iota_1: f\mapsto \widetilde{f}, 
$$
with
$$ \widetilde{f}(x,y)=f(x,y),\textrm{ if } |x|\leq 1,\textrm{ and } \widetilde{f}(x,y)=-f(2-x,y),\textrm{ if }1\leq x\leq 3.
$$
The mapping $\iota_1$ is the odd extension with respect to $x=1$. Note that for $f\in C^{\infty}(\ov{\Omega} )$, we have
$$ \partial_xf|_{x=1-}=\partial_x(\iota_1f)|_{x=1+}.
$$
\begin{lem}\label{embedding1}
	The extension map $\iota_1: D(\Delta_G)\rightarrow D(P_a)$ is continuous. Moreover, for all $f\in D(\Delta_G)$, $\|\iota_1f\|_{L^2(\widetilde{\T}^2)}=\sqrt{2}\|f\|_{L^2(\Omega)}$.
\end{lem}
\begin{proof}
	Assume that $u\in D(\Delta_G)$, we extend it to $\widetilde{u}=\iota_1u$. Denote by $\Omega'=(1,3)_x\times\T_y$. By definition, $(P_a\widetilde{u})|_{\Omega}\in L^2(\Omega)  $ and  $(P_a\widetilde{u})|_{\Omega'}\in L^2(\Omega')$. To check that $P_a\widetilde{u}\in L^2(\widetilde{T}^2)$, we have to be careful near $x=1$, where we glue $\Omega$ and $\Omega'$. Take a test function $F\in C^{\infty}(\widetilde{\T}^2)$, we calculate
	\begin{equation*}
	\begin{split}
	-\langle P_a\widetilde{u}, F\rangle=&\int_{\Omega\cup \Omega'}(\partial_x\widetilde{u}\partial_x F+a(x)^2\partial_y\widetilde{u}\partial_y F  )dxdy\\
	=&\int_{\Omega\cup \Omega'}(-\partial_x^2\widetilde{u}\cdot F-a(x)^2\partial_y^2u\widetilde{u} F)dxdy+\int_{x=1-}\partial_x\widetilde{u} F dy-\int_{x=-1+}\partial_x\widetilde{u} Fdy\\
	+& \int_{x=3-}\partial_x\widetilde{u} Fdy-\int_{x=1+}\partial_x\widetilde{u}Fdy. 
	\end{split}
	\end{equation*}
	From the definition of $\widetilde{u}$, we have $\partial_x\widetilde{u}|_{x=1+}=\partial_xu|_{x=1-}$, and $\partial_x\widetilde{u}|_{x=3-}=\partial_x\widetilde{u}_{x=-1+}$. Thus all the boundary terms vanish. This implies that
	$$ \|P_a\iota_1u\|_{L^2(\widetilde{\T}^2)}\leq \|\Delta_Gu\|_{L^2(\Omega)}+\|(P_a\widetilde{u})_{\Omega'}\|_{L^2(\Omega')}= 2\|\Delta_Gu\|_{L^2(\Omega)}.
	$$
	The last assertion $\|\iota_1f\|_{L^2(\widetilde{\T}^2)}=\sqrt{2}\|f\|_{L^2(\Omega)}$ is obvious. This completes the proof of Lemma~\ref{embedding1}.
\end{proof}

\begin{lem}\label{commutation}
	Let $S_1,S_2$ be two self-ajoint operators on Banach spaces $E_1, E_2$ with domains $D(S_1), D(S_2)$, with respectively. Assume that $j: D(S_1)\rightarrow D(S_2)$ is a continuous embedding. Suppose that there holds $j\circ S_1=S_2\circ j,$ then for any Schwartz function $g\in \mathcal{S}(\R)$, we have
	$$ j\circ g(S_1)=g(S_2)\circ j
	$$   	
\end{lem}
\begin{proof}
	Thanks to Helffer-Sj\"ostrand's formula, it suffices to check the validity for resolvent, namely
	\begin{equation}\label{resolvent-validity}
	j\circ (z-S_1)^{-1}=(z-S_2)^{-1}\circ j.
	\end{equation}
	For $f_1\in E_1$, we denote by $u_1=(z-S_1)^{-1}f_1\in D(S_1)$, and $u_2=j(u_1)\in E_2$. We directly check that
	$$ (z-S_2)u_2=(z-S_2)\circ j(u_1)=j((z-S_1)u_1 )=j(f_1).
	$$
	This completes the proof of Lemma \ref{commutation}.
\end{proof}
Lemma \ref{commutation}  ensures the preservation of the spectral localization property by odd extension procedure. Another interesting fact is that the extended eigenfunctions are still smooth. Indeed, since $\partial_y$ commutes with $\Delta_G$ as well as $P_a$, the eigenfunctions of $P_a$ can be taken of the form $\widetilde{\varphi}_n(x)e^{iny}$, where $\widetilde{\varphi}_n(x)$ is an eigenfunction of $\widetilde{L}_n:=-\partial_x^2+a(x)^2n^2$ with domain $D(\widetilde{L}_n)=H^2(\widetilde{\T})$.  Moreover, the extension $\iota_1$ can be viewed as a continuous map from $D(L_n)\rightarrow D(\widetilde{L}_n)$, where $L_n=-\partial_x^2+n^2x^2$ is defined on its domain $D(L_n)=H_0^1((-1,1))\cap H^2((-1,1))$. Recall that $(\varphi_{m,n})_{m\in\N}$ is a sequence of eigenfunctions of $L_n$ with associated eigenvalues $(\lambda_{m,n}^2)_{m\in\N}$. Therefore, the extension  $\widetilde{\varphi}_{m,n}:=\iota_1(\varphi_{m,n})$ is eigenfunction of $\widetilde{L}_n$ with the same eigenvalue $\lambda_{m,n}^2$. Note that $\widetilde{\varphi}_{m,n}$ is $C^{\infty}$ for $x\in (-1,1)\cup (1,3)$. From $L_n\varphi_{m,n}=\lambda_{m,n}^{2}\varphi_{m,n}$, we have $\varphi_{m,n}''(x)=(x^2n^2-\lambda_{m,n}^2)\varphi_{m,n}$. By induction, we have that $\varphi_{m,n}^{(2k)}|_{x=\pm 1}=0$, for all $k\in\mathbb{N}$. Since $\widetilde{\varphi}_{m,n}$ is odd with respect to $x=1$, we deduce that $\widetilde{\varphi}_{m,n}^{(2k+1)}|_{x=1-}=\widetilde{\varphi}_{m,n}^{(2k+1)}|_{x=1+}$, $\widetilde{\varphi}_{m,n}^{(2k+1)}|_{x=3-}=\widetilde{\varphi}_{m,n}^{(2k+1)}|_{x=-1+}$. Therefore, $\widetilde{\varphi}_{m,n}\in C^{\infty}(\widetilde{\T})$. Moreover, the orthogonality condition holds
$$ (\widetilde{\varphi}_{m,n},\widetilde{\varphi}_{m',n} )_{L^2(\widetilde{\T})}=C_n\delta_{m,m'}.
$$
Now we extend $(\widetilde{\varphi}_{m,n})_{m\in\N}$ to a orthonormal eigenbasis of $\widetilde{L}_n$. This new basis will be denoted by $(e_{m,n})_{m\in\N}$. Let $f_n(x)=\sum_{m\in\N}c_{m,n}\varphi_{m,n}(x)$, then for any Schwartz function $g:\R\rightarrow \C$, we have 
\begin{equation}\label{extended-localization}
(\iota_1\circ g(h^2L_n)f_n)(x)=\sum_{m\in\N}c_{m,n}g(h^2\lambda_{m,n}^2)\iota_1(\varphi_{m,n})(x).
\end{equation}
Using \eqref{extended-localization}, since the coefficient corresponding to the new added eigenfunctions $e_{m,n}$ is zero, we deduce that for any Schwartz function $g:\R\rightarrow \C$,
\begin{equation}\label{commutative}
\iota_1\circ g(h^2L_n)=g(h^2\widetilde{L}_n)\circ\iota_1,\textrm{ and } \iota_1\circ g(h^2\Delta_G)=g(h^2P_a)\circ\iota_1. 
\end{equation}
Consequently, we have the following lemma, reducing the proof of Proposition \ref{Observation-Lowfrequency} to the observability of the extended solutions:
\begin{lem}\label{observability:extension}
	Let $T>0$, assume that for any $0<h<h_0\ll 1, 0<\epsilon\ll 1$, the following observability holds true for all  $\widetilde{u}_0\in L^2(\widetilde{\T}^2)$:
	\begin{multline}\label{ob:extension}
	\|\psi(h^2P_a)\psi_0(h^{\epsilon}D_y)\widetilde{u}_0\|_{L^2(\widetilde{\T}^2)}^2\\
	\leq C_T\int_{-T}^T\|\phi(y)\psi(h^2P_a)\psi_0(h^{\epsilon}D_y)\widetilde{u}(t)\|_{L^2(\widetilde{\T}^2)}^2dt+C_Th\|\widetilde{u}_0\|_{L^2(\widetilde{\T}^2)}^2,
	\end{multline}
	where $\widetilde{u}(t)=e^{itP_a}\widetilde{u}_0$, the solution of Schr\"odinger equation $i\partial_t\widetilde{u}+P_a\widetilde{u}=0$ with initial data $\widetilde{u}|_{t=0}=\widetilde{u}_0$. Then Proposition \ref{Observation-Lowfrequency} is true. More precisely, with the same constant $C_T>0$, for all $0<h<h_0$, $0<\epsilon\ll 1$, the observability
	\begin{multline}\label{ob:finalregime}
	\|\psi(h^2\Delta_G)\psi_0(h^{\epsilon}D_y)u_0\|_{L^2(\Omega)}^2\\
	\leq C_T\int_{-T}^T\|\phi(y)\psi(h^2\Delta_G)\psi_0(h^{\epsilon}D_y)e^{it\Delta_G}u_0 \|_{L^2(\Omega)}^2dt+C_Th\|u_0\|_{L^2(\Omega)}^2
	\end{multline}
	holds true for all $u_0\in L^2(\Omega)$.
\end{lem} 

\begin{proof}
	With a little abuse of notation, we assume that $u_0=\psi(h^2\Delta_G)\psi_0(h^{\epsilon}D_y)u_0\in D(\Delta_G)$. Denote by $\widetilde{u}_0=\iota_1u_0$. Thanks to Lemma \ref{embedding1}, $\widetilde{u}_0=\psi(h^2P_a)\psi_0(h^{\epsilon}D_y)\widetilde{u}_0$. Using \eqref{commutative} with $g(r)=e^{itr}\psi(h^2r)$, we deduce that
	$$ \phi(y)e^{itP_a}\widetilde{u}_0=\iota_1(\phi(y)e^{it\Delta_G}u_0).
	$$	
	From the observability for $\widetilde{u}_0$, we obtain that
	$$ \|\widetilde{u}_0\|_{L^2(\widetilde{\T}^2)}^2\leq C_T\int_{-T}^T\|\phi(y)e^{itP_a}\widetilde{u}_0\|_{L^2(\widetilde{\T}^2)}^2dt,
	$$
	which implies that
	$$ \|u_0\|_{L^2(\Omega)}^2\leq C_T\int_{-T}^T\|\phi(y)e^{it\Delta_G}u_0\|_{L^2(\Omega)}^2dt,
	$$
	thanks to the fact that $\|\iota_1f\|_{L^2(\widetilde{\T}^2)}^2=2\|f\|_{L^2(\Omega)}^2.$ This completes the proof of Lemma \ref{observability:extension}.
\end{proof}

\begin{proof}[Proof of Proposition \ref{Observation-Lowfrequency}]
	From Lemma \ref{observability:extension}, it is sufficient to prove  \eqref{ob:extension}. With a little abuse of notation, we denote by $u_0\in D(P_a)$ such that $u_0=\psi(h^2P_a)\psi(h^{\epsilon}D_y)u_0$ and $u(t)=e^{itP_a}u_0$. We are now in the periodic setting. Yet, we should pay an extra attention to the fact that $P_a=\partial_x^2+a(x)^2\partial_y^2$ is a hypoelliptic operator with only Lipschitz coefficient. More precisely, $a\in \mathrm{Lip}(\widetilde{\T}^2)$ which is not $C^1$ at $x=1$.

	Before proceeding, we need a lemma which, modulo errors allows us to replace the microlocalisation $\psi(h^2P_a)\psi_0(h^{\epsilon}D_y)$ by $\psi_1(hD_x)\psi(h^{\epsilon}D_y)$.
	\begin{lem}\label{changing-localization}
		Let $\psi_1\in C_c^{\infty}(\frac{1}{4}<|\xi|<4)$ such that $\psi_1=1$ on supp$(\psi)$. Then, as bounded operator on $L^2(\widetilde{\T}^2)$, we have
		$$ \left(1-\psi_1(hD_x) \right)\psi(h^2P_a)\psi_0(h^{\epsilon}D_y)=O_{L^2\rightarrow L^2}(h^{\frac{3}{2}-\frac{5}{2}\epsilon}).
		$$ 	
	\end{lem}
	\begin{proof}
		Since $D_y$ commutes with everything, it suffices to show that, uniformly in $|n|\leq Ch^{-\epsilon}$, as an operator on $L^2(\widetilde{\T})$,
		\begin{equation}\label{*}
		(1-\psi_1(hD_x) )\psi(h^2\widetilde{L}_n)=O_{L^2\rightarrow L^2}(h^{\frac{3}{2}-\frac{5}{2}\epsilon}).
		\end{equation}
		Let $f=\psi(h^2\widetilde{L}_n)f$,  $f(x)=\sum_{j}\psi(h^2\lambda_{j,n}^2)c_{j}e_{j,n}(x)$. Note that
		$$ -(h^2\partial_x^2+h^2\lambda_{j,n}^2)e_{j,n}=-h^2n^2a(x)^2e_{j,n}=O_{L^2}(h^{2-2\epsilon}).
		$$
		Observe that on the support of $(1-\psi_1(\xi))$, the semi-classical operator $-(h^2\partial_x^2+h^2\lambda_{j,n}^2)$ is elliptic, uniformly in $j,n$ such that $|n|\leq Ch^{-\epsilon}$ and $\lambda_{j,n}^2h^2\in $ supp$(\psi).$ Therefore, $\|(1-\psi_1(hD_x))e_{j,n}\|_{L^2(\widetilde{\T})}=O(h^{2-2\epsilon})$. By Cauchy-Schwartz,
		$$ \|(1-\psi_1(hD_x))f\|_{L^2(\widetilde{\T})}\leq \sup_{j}\|(1-\psi_1(hD_x))e_{j,n}\|_{L^2(\widetilde{\T})}\cdot \#\{j: \lambda_{j,n}^2\in\textrm{ supp }(\psi) \}^{1/2}\|f\|_{L^2(\widetilde{\T})}.
		$$
		From Weyl's law,  $$\#\{j: \lambda_{j,n}^2\in\textrm{supp}(\psi) \}\leq Ch^{-1-\epsilon}\leq \#\{(j,n): |n|\leq Ch^{-\epsilon}, \lambda_{j,n}^2\in\textrm{ supp }(\psi) \}\leq Ch^{-1-\epsilon}. $$ 
		Therefore, 
		$$ \|(1-\psi_1(hD_x))f \|_{L^2(\widetilde{\T})}\leq Ch^{\frac{3}{2}-\frac{5}{2}\epsilon}\|f\|_{L^2(\widetilde{\T})}.
		$$
		Applying Plancherel in $y$, we complete the proof of Lemma \ref{changing-localization}.
	\end{proof}
	Modulo an error $O_{L^2}(h^{\frac{3}{2}-\frac{5}{2}\epsilon})\|u_0\|_{L^2(\widetilde{\T}^2)}$, we may assume that $u=\psi_1(hD_x)\psi_0(h^{\epsilon}D_y)u$. Now we search for the function
	$$ v=(1+hQD_y^2)u
	$$
	with a operator $Q$ acting only in $x$, to be chosen later. Let $M=\frac{1}{4}\int_{-1}^3 a(x)^2dx= \frac 1 2 \int_{-1}^1 x^2 dx$ be the average of $a(x)^2$ along the horizontal trajectory $y=\mathrm{const.}$ Using the equation $(i\partial_t+P_a)u=0$, we have
	\begin{equation*}
	\begin{split}
	(i\partial_t-D_x^2-MD_y^2)v=&(1+hQD_y^2)(a(x)^2-M)D_y^2u-h[D_x^2,Q]D_y^2u\\
	=&(a(x)^2-M)D_y^2u-h[D_x^2,Q]D_y^2u+hQD_y^2(a(x)^2-M )D_y^2u \\
	=&(a(x)^2-M)D_y^2u-h[D_x^2,Q]D_y^2u + O_{L^2(\widetilde{\T}^2)}(h^{1-4\epsilon}),
	\end{split}
	\end{equation*}
	since $\|hQD_y^2(a(x)^2-M)D_y^2u\|_{L^2(\widetilde{\T}^2)}=O(h^{1-4\epsilon})$, due to the localization property of $u$. 
	Take $\psi_2\in C_c^{\infty}(1/8\leq |\xi|\leq 8)$, such that $\psi_2\psi_1=\psi_1$. We define the operator
	$$ Q=\frac{1}{2 i}\left(\int_{-1}^x(a(z)^2-M )dz\right)(hD_x)^{-1}\psi_2(hD_x),
	$$ 
	and denote by $b(x)=\frac{1}{2 i}\int_{-1}^x(a(z)^2-M )dz$, $m(hD_x)=(hD_x)^{-1}\psi_2(hD_x)$. Since $a(x)^2-M$ has zero average, the function $b$ is well-defined as a periodic function in the space $C^1(\widetilde{\T})\cap W^{2,\infty}(\widetilde{\T})$. From direct calculation, we have
	$$ -h[D_x^2,Q]=2ib'(x)m(hD_x)hD_x+i[hD_x,b'(x)]m(hD_x).
	$$
	Note that $[hD_x,b'(x)]=hb''(x)$, and $b''\in L^{\infty}(\widetilde{\T})$, thus
	$$ (i\partial_t+\Delta_M)v=r_h=O_{L^2(\widetilde{\T}^2)}(h^{1-4\epsilon})\|u_0\|_{L^2(\widetilde{\T}^2)},
	$$
	where $\Delta_M=\partial_x^2+M\partial_y^2$.
	Applying Theorem \ref{ob:Schrodinger}, we obtain that
	\begin{equation*}
	\begin{split}
	\|v(0,\cdot)\|_{L^2(\widetilde{\T}^2)}^2\leq &C_T\int_{-T}^T\|\phi(y)v(t,\cdot) \|_{L^2(\widetilde{\T}^2)}^2dt+C_T\int_{-T}^T\left\|\phi(y)\int_0^te^{i(t-t')\Delta_M}r_h(t')dt' \right\|_{L^2(\widetilde{\T}^2)}^2dt\\
	\leq & C_T\int_{-T}^T\|\phi(y)v(t,\cdot)\|_{L^2(\widetilde{\T}^2)}^2dt+C_Th^{2(1-4\epsilon)}\|u_0\|_{L^2(\widetilde{\T}^2)}^2.
	\end{split}
	\end{equation*}
Since $v=u+O_{L^2(\widetilde{\T}^2)}(h^{1-2\epsilon})\|u_0\|_{L^2(\widetilde{\T}^2)}$, the proof of Proposition \ref{Observation-Lowfrequency} is now complete.
\end{proof}

\section{Compactness argument and the proof of Theorem \ref{positive-bounded}}

\begin{lem}\label{embedding}
	The embeddings  $H_{G,0}^1(\Omega)\hookrightarrow L^2(\Omega)$ and	$L^2(\Omega)\hookrightarrow H_{G,0}^{-1}(\Omega)$ are compact.
\end{lem}
\begin{proof}
	By duality, we only need to prove that $H_{G,0}^1(\Omega)\hookrightarrow L^2(\Omega)$ is compact. Denote by $\Pi_0$, the projection to the zero mode of $y$. Define several closed linear subspaces:
	$$ \mathcal{H}_1:=\{f\in L^2(\Omega): \mathcal{F}_yf(\cdot,0)=0 \}, \quad \mathcal{H}_2:=\Pi_0(L^2(\Omega)),$$
	and
	$$ X_1:=\{f\in H_{G,0}^1(\Omega): \mathcal{F}_yf(\cdot,0)=0 \},\quad X_2:=\Pi_0(H_{G,0}^1(\Omega)).
	$$
	Then 
	$$ L^2(\Omega)=\mathcal{H}_1\oplus \mathcal{H}_2, \quad H_{G,0}^1(\Omega)=X_1\oplus X_2.
	$$
	Moreover, $X_j\subset\mathcal{H}_j,j=1,2$ and $\mathcal{H}_1\perp \mathcal{H}_2$, with respect to the $L^2(\Omega)$ inner product. The compact embedding will follow, if we prove that the each embedding $i_j: \mathcal{H}_j\hookrightarrow X_j$ is compact. Note that $X_2$ is isometric to $H_0^1((-1,1))$, while $\mathcal{H}_2$ is isometric to $L^2((-1,1))$. By Rellich theorem, $i_2$ is compact. 
	
	It remains to show that $i_1$ is compact. Denote again by $\Pi_+$ and $\Pi_-$ the projection to strictly positive and negative frequencies, with respectively. For $f\in X_1$, we can write
	$$ \||D_y|^{1/2}u\|_{L^2(\Omega)}^2=(-i\partial_y\Pi_+f,f )_{L^2(\Omega)}-(-i\partial_y\Pi_- f,f)_{L^2(\Omega)}^2.
	$$ 
	Using $[\partial_x,x\partial_y ]=\partial_y$, we deduce that $$(\partial_y\Pi_{\pm}f,f)_{L^2(\Omega)}\leq \|\partial_x f\|_{L^2(\Omega)}\|x\partial_yf\|_{L^2(\Omega)}\leq \|u\|_{H_{G}^1(\Omega)}^2. $$
	Since $\|\partial_xf\|_{L^2(\Omega)}$ can be controlled by $\|f\|_{H_{G}^1(\Omega)}$, we deduce that $H_{G,0}^1(\Omega)\hookrightarrow H^{1/2}(\Omega)$ is continuous. Decomposing $i_1$ as 
	$ H_{G,0}^1(\Omega)\hookrightarrow H^{1/2}(\Omega)\hookrightarrow L^2(\Omega), 
	$
	the compactness of $i_1$ is implied by the compactness of $H^{1/2}(\Omega)\hookrightarrow L^2(\Omega)$. This completes the proof of Lemma \ref{embedding}. 
\end{proof}

\begin{proof}[Proof of Theorem \ref{positive-bounded}]
	Combining with Proposition \ref{regime:Half-wave}, Proposition \ref{Regime:GCC}, Proposition \ref{ob:Regime2} and Proposition \ref{Observation-Lowfrequency}, for some $h_0>0$, for $T>a$ and all $0<h<h_0$, there holds
	$$ \|\psi(h^2\Delta_G)u_0\|_{L^2(\Omega)}^2\leq C_T\int_{-T}^{T}\|\phi(y)e^{it\Delta_G}\psi(h^2\Delta_G)u_0\|_{L^2(\Omega)}^2dt+C_Th^{1/2}\|u_0\|_{L^2(\Omega)}^2.
	$$ 
	The rest argument is now standard. We follow the approach of Bardos-Lebeau-Rauch \cite{BLR}.
	Choosing $h=2^{-j}$ and summing over the inequality above for $j\geq j_0=\log_2\left(h_0^{-1}\right)$, after standard manupulations, we obtain that
	\begin{equation}\label{ob:weak}
	\|u_0\|_{L^2(\Omega)}^2\leq C_T\int_{-T}^T\|\phi(y)e^{it\Delta_G}u_0\|_{L^2(\Omega)}^2dt+C_T\|\psi_0(c2^{-2j_0}\Delta_G)u_0\|_{L^2(\Omega)}^2,
	\end{equation}
	where $\psi_0\in C_c^{\infty}(|\xi|\leq 1)$. Note that the second term on the right side is controlled by $\|u_0\|_{H_G^{-1}(\Omega)}^2$.
	
	For $T'>0$, defining the set
	$$ \mathcal{N}_{T'}:=\left\{u_0\in L^2(\Omega): e^{it\Delta_G}u_0|_{[-T',T']\times\omega}=0 \right\}
	$$
	Take $T'\in (a,T)$,  \eqref{ob:weak} implies that any function $u_0\in \mathcal{N}_{T'}$ satisfies
	$$ \|u_0\|_{L^2(\Omega)}\leq C_T\|u_0\|_{H_G^{-1}(\Omega)}.
	$$
	Thanks to Lemma \ref{embedding}, we deduce that $\mathrm{dim}\mathcal{N}_{T'}<\infty$. Note that for any $T_1<T_2$, $\mathcal{N}_{T_2}\subset \mathcal{N}_{T_1}$. Consider the mapping  $S(\delta):=\delta^{-1}\left(e^{i\delta\Delta_G}-\mathrm{Id}\right): \mathcal{N}_{T'}\rightarrow \mathcal{N}_{T'-\delta}$. For fixed $T'\in (a,T)$, when $\delta< T'-a$, $\dim\mathcal{N}_{T'-\delta}<\infty$. Since the dimension is an integer, there exists $\delta_0>0$, such that for all $0<\delta<\delta_0$, $\mathcal{N}_{T'-\delta}=\mathcal{N}_{T'}$. Therefore, $S(\delta)$ is a linear map on $\mathcal{N}_{T'}$. Let $\delta\rightarrow 0$, we obtain that $\partial_t: \mathcal{N}_{T'}\rightarrow \mathcal{N}_{T'}$ is a well-defined linear operator. Take any eigenvalue $\lambda\in\C$ of $\partial_t$, and assume that $u_*\in\mathcal{N}_{T'}$ is a corresponding eigenfunction (if it exists). There holds
	$$ -\Delta_Gu_{*}=i\lambda u_{*}.
	$$ 
	This implies that $u_*$ is an eigenfunction of $-\Delta_G$. However, $u_*|_{\omega}\equiv 0$, hence $u_*\equiv 0$. Therefore, $\mathcal{N}_{T'}=\{0\}$.
	
	Now we choose $T_0=T'$ as above. By contradiction, assume that Theorem \eqref{positive-bounded} is untrue. Then there exists a sequence $(u_{k,0})_{k\in\N}$, such that 
	$$ \|u_{k,0}\|_{L^2(\Omega)}=1,\quad \lim_{k\rightarrow\infty}\int_{-T_0}^{T_0}\|\phi(y)e^{it\Delta_G}u_{k,0}\|_{L^2(\Omega)}^2=0.
	$$
	Up to a subsequence, we may assume that $u_{k,0}\rightharpoonup u_0$. Thus from Lemma \ref{embedding}, $u_{k,0}\rightarrow u_0$, strongly in $H_G^{-1}(\Omega)$. Passing $k\rightarrow\infty$ of \eqref{ob:weak}, $1\leq \|u_0\|_{H_G^{-1}(\Omega)}$. In particular, $u_0\neq 0$. However, since $e^{it\Delta_G}u_{k,0}\rightarrow 0$ in $L^2([-T_0,T_0]\times\omega)$. This implies that $u_0\in\mathcal{N}_{T_0}=\{0\}$. This is a contradiction. The proof of Theorem \ref{positive-bounded} is now complete. 
\end{proof}

\section{Counterexample to the observability if \texorpdfstring{$T\leq \mathcal{L}(\omega)$}{T<L(omega)}} 



The goal of this section is to prove Theorem \ref{thm-negative}. This is based on detailed asymptotic analysis for the first eigenvalue and eigenfunction of the harmonic oscillator $\mathcal{L}_w=-\partial_x^2+w^2x^2$ with respect to the large parameter $w$. Some of the estimates can be also found in \cite{HS}. We will present an self-contained ODE approach in this section.  First, we note that it is sufficient to consider the case $T<a$. Since if the interior observability holds true for $T=a$, then there exists $\delta_1>0$, such that the observability is also true for $T=a-\delta$.\footnote{This comment would not apply for a boundary observability problem. }

By changing of variable $z=\sqrt{w}x$, it is more convenient to deal with the harmonic oscillator $\mathcal{P}=-\partial_z^2+z^2$ on $L^2((-\sqrt{w},\sqrt{w}))$ with Dirichlet boundary condition.

Let us consider the $\mu$ dependent equations
\begin{equation}
\begin{cases}\label{Hermite-1}
& -f''+(z^2-\mu)f=0,\\
& f(0)=1,\quad f'(0)=0.
\end{cases}
\end{equation}
We will denote by $f_{\mu}(z)=f(\mu,z)$ the unique solution to \eqref{Hermite-1}. 

Denote by
$$ Z_0(\mu):=\inf\{z>0: f_{\mu}(z)=0 \}.
$$
$Z_0(\mu)$ is the first zero of $f(\mu,z)$. If $f_{\mu}$ does not have zero, then $Z_0(\mu)=+\infty$. By inverse function theorem and Strum-Liouville theorem, it is known that $Z_0(\mu)<1$ for all $\mu>1$, and the function $\mu\mapsto Z_0(\mu)$ is strictly decreasing for $\mu>1$. Moreover, $Z_0(1)=+\infty$.
Denote by $w\mapsto \mu(w)$ the inverse function of $Z_0(\mu)$. It is also known that the function $\mu(w)$ is smooth, strictly decreasing, and for all $w>0$,
	$$ \mu(w)>1,\quad  \lim_{w\rightarrow\infty} \mu(w)=1.
	$$	
Since $\mu(w)-1$ tends to $0$ as $w\rightarrow \infty$, it is more convenient to work with
$$ \nu(w):=\mu(w)-1
$$
and denote by $f(\nu,z)=f(\mu,z)$, the solution of \eqref{Hermite-1}, by abusing the notation.
\subsection{Lack of observability for \texorpdfstring{$T<\mathcal{L}(\omega)$}{T< L(omega)} }
Let $I=(-1,1)$, and denote by
$$ p(w,x):=\frac{f(\nu(w),\sqrt{w}x)}{\|f(\nu(w),\sqrt{w}\cdot)\|_{L^2(I)}}
$$
the first normalized eigenfunction of $\mathcal{L}_{w}:=-\frac{d^2}{dx^2}+w^2x^2$ on $L^2((-1,1))$, which is real-valued, with the least eigenvalue $\lambda(w)=w(1+\nu(w))$. For $m\in\N, y\in\R$, we define the phase function
$$ \Phi_m(t,y,w)=w y-\lambda(w)t-2\pi m w.
$$
The following estimate for the eigenvalue $\lambda(w)$, eigenfunction $p(w,x)$ and phase function $\Phi_m(t,y,w)$ is crucial for the proof of Theorem \ref{thm-negative}.
\begin{prop}\label{estimate:eigenvalue}
	There exists $w_0>0$, large enough, such that for all $w\geq w_0$, $|x|\leq \frac{1}{2}$, 
	\begin{equation*}
	\begin{split}
	(1)\quad &|\lambda(w)-w|\leq C_0w^{3/2}e^{-w}, \nu^{(k)}(w)\leq C_kw^{A_k}e^{-w}, \textrm{ for } k \geq 0;\\
	(2)\quad &|\lambda^{(k)}(w)|\leq C_kw^{A_k}e^{-w}, \textrm{ for } k\geq 2;\\ (3)\quad &\Big|\frac{\partial^k\Phi_m}{\partial w^k}(t,y,w)\Big|\leq C_0(1+|t|)w^{A_k}e^{-w}, \textrm{ for } k\geq 2, t,y\in\R, m\in\N;\\
	(4)\quad  & cw^{1/4}e^{-\frac{wx^2}{2}}\leq  |p(w,x)|\leq w^{1/4}e^{-\frac{wx^2}{2}}\textrm{ and } \Big|\frac{\partial^kp}{\partial w^k}(w,x)\Big|\leq Cw^{1/4-k}, \textrm{ for }k=1,2;\\
	(5)\quad & |p(w,x_0)|\leq e^{-\frac{w}{10}}, \textrm{ for } \frac{1}{2}<|x_0|\leq 1.
	\end{split}
	\end{equation*}
	where $A_k,C_k$ are positive constants depending on $k$ and $C_0,C,c>0$ are uniform constants.
\end{prop}
\begin{rem}
The content of Proposition \ref{estimate:eigenvalue} is not new. Indeed, estimates for the eigenvalues and eigenfunctions of more general semi-classical harmonic oscillator can be found in \cite{HS}. However, for our specific purpose, and to have a self contained exposition, we will present an elementary by-hand proof of the Proposition \ref{estimate:eigenvalue}, which may be of its own interest.
\end{rem}
We postpone the proof of Proposition \ref{estimate:eigenvalue} and use it to prove Theorem \ref{thm-negative}. It is sufficient to consider the case where $\omega^c$ consists of a single band. Thus a reformulation of \ref{thm-negative} is the following:
\begin{prop}\label{final-negative}
	Let $T<a$. Then there exists a sequence of solutions $(u_n)_{n\in \N}$ of equation \eqref{main-equation-bounded}, such that
	$$ \liminf_{n\rightarrow\infty}\int_{\Omega}|u_n(0,x,y)|^2dxdy>0,
	$$
	while
	\begin{equation}\label{eq:counter-ob} \lim_{n\rightarrow\infty}\int_{-T}^T\int_{(-1,1)_x\times ((a,\pi)\cup (-\pi,-a))_y}|u_n(t,x,y)|^2dxdydt=0.
	\end{equation}
\end{prop}
\begin{proof}
	We use the similar Gaussian beam construction as in \cite{RS18}. Let $h_n=2^{-n}$ and define $$ g_n(w)=\frac{1}{2\pi}\int_{-\pi}^{\pi}e^{-\frac{y^2}{2h_n^2}-iwy}\frac{dy}{\sqrt{h_n}},\quad \chi\in C_c^{\infty}((1/2,2)).
	$$
By changing variables
	we observe that $\|g_n\|_{L^{\infty}}=O(h_n^{1/2})$ and $\|\partial_w^kg_n\|_{L^{\infty}}=O(h_n^{k+\frac{1}{2}})$. Define
	$$ u_{n,0}(x,y)=\sum_{k\in\Z} g_n(k)\chi(h_nk)p(k,x)e^{iky}.
	$$
	We first claim that $\|u_{n,0}\|_{L^2(\Omega)}\sim 1$.
	Indeed, 
	$$ g_n(w)=\frac{h_n^{1/2}}{2\pi}\int_{-\frac{\pi}{h_n}}^{\frac{\pi}{h_n}}e^{-\frac{z^2}{2}-izh_nw}dz=\frac{h_n^{1/2}}{2\pi}\widehat{(e^{-\frac{|\cdot|^2}{2}})}(h_nw)-\frac{h_n^{1/2}}{2\pi}\int_{|z|>\frac{\pi}{h_n}}e^{-\frac{z^2}{2}-izh_nw}dz,
	$$
	hence $g_n(w)=\frac{h_n^{1/2}}{\sqrt{2\pi}}e^{-\frac{h_n^2w^2}{2}}+O(e^{-\frac{\pi^2}{4h_n^2}})$.
	By Plancherel and (4) of Proposition \ref{estimate:eigenvalue}, for fixed $x\in(-1/2,1/2)$,
	\begin{equation*}
	\begin{split}
	\|u_{n,0}(x,\cdot)\|_{L^2(\T)}^2=&\sum_{k\in\Z} |g_n(k)\chi(h_nk)|^2|p(k,x)|^2\\
	\geq &c\sum_{h_n^{-1}/2\leq |k|\leq 2h_n^{-1}} h_n e^{-h_n^2k^2}k^{1/2}e^{-\frac{kx^2}{2}}.
	\end{split}
	\end{equation*}
 Integrating with respect to $x\in(-1,1)$, we obtain that
	$$\|u_{n,0}\|_{L^2(\Omega)}^2\geq c e^{-4}\int_{-1/2}^{1/2}e^{-\frac{kx^2}{2}}k^{1/2}dx\geq c_1>0,
	$$
	uniformly in $n\in\N$.
	
	Next, we prove \eqref{eq:counter-ob}. Note that the solutions to \eqref{main-equation-bounded} with initial data $u_{n,0}$ are given by
	$$
	u_n(t,x,y)=\sum_{k\in \Z}g_n(k)\chi(h_nk)p(k,x)e^{iky-i\lambda(k)t}.
	$$
	If $|x|>\frac{1}{2}$, from (5) of Proposition \ref{estimate:eigenvalue},
	\begin{equation}\label{proof:negative1}
	 |u_n(t,x,y)|\leq Ch_n^{-1/2}e^{-\frac{1}{20h_n}}.
    \end{equation}
	For $|x|\leq \frac{1}{2}$, we estimate $|u_n(t,x,y)|$ in another way. From Poisson summation formula, we have
	$$ u_n(t,x,y)=\sum_{m\in\Z}\widehat{K_{t,x,y}^{(n)}}(2\pi m),
	$$
	where
	$$ \widehat{K_{t,x,y}^{(n)}}(2\pi m)=\int_{\R}g_n(w)\chi(h_n w)p(w,x)e^{i\Phi_m(t,y,w)}d w.
	$$
	Writting $e^{i\Phi_m}=\frac{1}{i\partial_w\Phi_m}\frac{\partial}{\partial w}(e^{i\Phi_m})$, we have from the integration by part that
	\begin{equation*}
	\begin{split}
	\widehat{K_{t,x,y}^{(n)}}(2\pi m)=& -\int_{\R}e^{i\Phi_m}\frac{\partial}{\partial w}\left[\frac{\partial_w\left(g_n(w)\chi(h_nw)p(w,x) \right)\partial_w\Phi_m-g_n(w)\chi(h_nw)p(w,x)\partial_w^2\Phi_m  }{|\partial_w\Phi_m(t,y,w)|^3} \right]dw.
	\end{split}
	\end{equation*}
	Thanks to Proposition \ref{estimate:eigenvalue}, $|\partial_w\Phi_m(t,y,w)|=|y-t-2\pi m|+o(e^{-w/2})$, $|\partial_w^k\Phi_m|=O(e^{-w/2})$, for $k\geq 2$. Therefore, if $|t|\leq T<a$ and $y\in(a,\pi)\cup (-\pi,-a),$ we have $|y-t-2\pi m|\geq |m-c_0|$ for some $0<c_0<1$, thus $|\partial_w\Phi_m|\neq 0$, since $w\sim h_n^{-1}$ is large. The main contributions inside the integral is 
	$$  \frac{\partial_w^2 (g_n(w)\chi(h_nw)p(w,x) )}{|\partial_w\Phi_m|^2}
	$$
	and all other terms can be bounded by $Ce^{-\frac{c}{h_n}}|y-t-2\pi m|^{-3}$. 
	
	\noindent The contributions of $\partial_w^2(g_n(w)\chi(h_nw)p(w,x) )$ are
	$$ O(h_n^{5/2}w^{1/4}),O(h_n^{3/2}w^{-3/4}), O(h_n^{1/2}w^{-7/4}).
	$$
	The integration is taken over $w\sim h_n^{-1}$, we have that
	\begin{equation}\label{proof:negative2}
	\sup_{(t,x,y)\in (0,T)\times\omega}|\widehat{K_{t,x,y}^{(n)}}(2\pi m)|\leq \frac{Ch_n^{5/4}}{|m-c_0|^2}.
\end{equation}
	Since $\sum_{m\in\Z}|m-c_0|^2<\infty$ , combining \eqref{proof:negative1} and \eqref{proof:negative2}, we conclude that
	$$ \lim_{n\rightarrow\infty}\|u_n(t,x,y)\|_{L^2((-T,T)\times\omega)}=0.
	$$
	The proof of Proposition \ref{final-negative} is now complete.
\end{proof}
In the rest subsections, we analyze the solution of the Hermite equation \eqref{Hermite-1} and prove Proposition \eqref{estimate:eigenvalue}.
\subsection{Formal power expansions}
 To solve the Hermite equation \eqref{Hermite-1}, we make the ansatz of the power series expansion 
$$ f(\nu,z)\simeq\sum_{j=0}^{\infty}\nu^jf_j(z)
$$ 
Plugging into \eqref{Hermite-1} and comparing the degree of $\nu^j$ in both side, we have 
\begin{equation*}
\begin{split}
&f_0''=(z^2-1)f_0,\quad f_0(0)=1,f_0'(0)=0,\\
& f_{n+1}''=(z^2-1)f_{n+1}-f_n,\quad f_{n+1}(0)=f_{n+1}'(0)=0,\forall n\geq 0.
\end{split}
\end{equation*}
It is easy to verify that
 $f_0(z)=e^{-\frac{z^2}{2}}$. Moreover, the formally expansion formula holds
 $$ f(\nu,z)\simeq\sum_{j=0}^{\infty} \nu^j f_j(z)\simeq e^{-\frac{z^2}{2}}+\sum_{j=1}^{\infty}(-1)^jc_j\nu^j\frac{(\log z)^{j-1}}{z}e^{\frac{z^2}{2}},
 $$
 where $c_j>0$ are numerical constants. However, for our need, we only need expand to the first order term $f_1(z)$.$$ f(\nu,z)=f_0(z)+\nu f_1(z)+\nu^2 R_2(\nu,z).
 $$ 
 The remainder term $R_2$ solves the equation
 \begin{equation}\label{R2}
 \begin{cases}
 & R_2''=(z^2-\nu-1)R_2-f_1,\\
 & R_2(0)=R_2'(0)=0.
 \end{cases}
 \end{equation}
For $f_1(z)$, we have
\begin{prop}\label{first-order}
	$ f_1(z)<0 
	$ for all $z>0 $ and 
	$$ f_1(z)=e^{\frac{z^2}{2}}\left(-\frac{\sqrt{\pi}}{4z}+O(z^{-3})\right),\quad z\rightarrow+\infty.
	$$
	Furthermore, for all $k\geq 0$, 
	$$ |f_1^{(k)}(z)|\leq C_k|z|^{k-1}e^{\frac{z^2}{2}},\quad \textrm{ as }z\rightarrow\infty.
	$$
\end{prop}

\begin{lem}\label{elementary}
	Assume that $f\in C^{1}([0,\infty))$ and $|f(z)|+|f'(z)|\leq C|z|^{N}$ for some $N\in\N^*$.
	Let $I_f(z)=\int_0^zf(y)e^{-(z^2-y^2)}dy$, then 
	$$ I_f(z)=\frac{f(z)}{2z}+O\left(\frac{1+|zf'(z)|}{|z|^3}\right) , \quad z\rightarrow +\infty.
	$$
\end{lem}
\begin{proof}
	By changing of variable: $u=y^2$, we have
	$$ I_f(z)=\int_0^{z^2}f(\sqrt{u})\frac{e^{-(z^2-u)}}{2\sqrt{u}}du=\int_0^{z^2}\frac{f(\sqrt{z^2-v})e^{-v}}{2\sqrt{z^2-v}}dv.
	$$
	By further changing of variable $v\mapsto z^2s^2$, we have
	\begin{equation*}
	\begin{split}
	I_f(z)=z\int_0^1\frac{sf(z\sqrt{1-s^2}) e^{-z^2s^2}}{\sqrt{1-s^2}}ds=z\int_0^{\frac{1}{2}}\frac{sf(z\sqrt{1-s^2})}{\sqrt{1-s^2}}e^{-z^2s^2}ds+o(e^{-z^2/8})
	\end{split}
	\end{equation*}
	since $\frac{s}{\sqrt{1-s^2}}$ is integrable on $(1/2,1)$ and $f$ is polynomially growth in $z$. Writting the first integral on the right hand side as $-\frac{1}{z}\int_0^{\frac{1}{2}}\frac{f(z\sqrt{1-s^2})}{2\sqrt{1-s^2}}\frac{d}{ds}(e^{-z^2s^2})ds$ and doing integration by part, we have
	\begin{equation*}
	\begin{split}
	&-\frac{1}{z}\int_0^{\frac{1}{2}}\frac{f(z\sqrt{1-s^2})}{2\sqrt{1-s^2}}\frac{d}{ds}(e^{-z^2s^2})ds\\=&\frac{f(z)}{2z}+o(e^{-z^2/8})+\frac{1}{2z}\int_0^{\frac{1}{2}}e^{-z^2s^2}\left( f(z\sqrt{1-s^2})(1-s^2)^{-\frac{1}{2}} \right)'ds\\
	=&\frac{f(z)}{2z}+O(e^{-z^2/4})-\frac{1}{4z^3}\int_0^{\frac{1}{2}}\frac{d}{ds}(e^{-z^2s^2}) \frac{1}{s}\left(f(z\sqrt{1-s^2})(1-s^2)^{-\frac{1}{2}} \right)'ds.
	\end{split}
	\end{equation*}
	Since the third term on the right hand side is $O(z^{-3})$, the proof of Lemma \ref{elementary} is complete.
\end{proof}

\begin{proof}[Proof of Proposition \ref{first-order}]
	Denote by $g_1(z)=e^{-\frac{z^2}{2}}f_1(z):=f_1f_0$, then $g_1$ solves the equation 
	\begin{equation}\label{g1}
	\begin{cases}
	& g_1''+2(zg_1)'=-e^{-z^2},\\
	&g_1(0)=g_1'(0)=0,\quad g_1''(0)<0.
	\end{cases}
	\end{equation}	
	Integrating \eqref{g1} and then solving a first order ODE, we obtain that
	$$ g_1(z)=-\int_0^z \Phi(y)e^{-(z^2-y^2)}dy<0, \textrm{ where }\Phi(z)=\int_0^ze^{-y^2}dy.
	$$
	Since $\Phi(z)$ is bounded as well as its derivatives and $g_1(z)=-I_{\Phi}(z)$, applying Lemma \ref{elementary}, we obtain that
	$ g_1(z)=-\frac{\Phi(z)}{2z}+O(|z|^{-3})
	,$ as $z\rightarrow +\infty$. From the fact that $|\frac{\sqrt{\pi}}{2}-\Phi(z)|=|\Phi(+\infty)-\Phi(z)|=O(e^{-z^2/2}),$ as $z\rightarrow+\infty$, we obtain that $g_1(z)=-\frac{\sqrt{\pi}}{4z}+O(|z|^{-3})$.
	The estimate for the derivatives can be deduced by induction. Since $f_1'=e^{\frac{z^2}{2}}(g_1'+zg_1)$ and $g_1'(z)=-\Phi(z)+\int_0^z\Phi(z)ze^{-(z^2-y^2)}dy$, we have $|f_1'(z)|\leq C_1e^{\frac{z^2}{2}}$, thanks to Lemma \ref{elementary}. Finally, from the equation of $f_1$, we deduce that $|f^{(k)}(z)|\leq C|z|^{k-1}e^{\frac{z^2}{2}}$ and this completes the proof of Proposition \ref{first-order}. 
\end{proof}

\subsection{Estimate for remainders: Laplace method}
Our goal is to obtain estimates for $R_2$ and its derivatives as $z\rightarrow\infty$. It is difficult to write the solution of \eqref{R2} explicitly. Fortunately, since \eqref{R2} is a (inhomogeneous) Hermite type equation, we could obtain the asymptotics by using the Laplace transforamtion. For this standard method, we follow the book \cite{Wang-Guo}. 

 Let $P_2(\nu,z)=R_2(\nu,z)e^{\frac{z^2}{2}}, Q_2(z)=-f_1(z)e^{\frac{z^2}{2}}$, then \eqref{R2} is transformed to the following equation:
\begin{equation}\label{P2}
\begin{cases}
& P_2''-2zP_2'+\nu P_2=Q_2,\\
&P_2(0)=P_2'(0)=0.
\end{cases}
\end{equation}
We first look for two independent solutions of the homogeneous equation
\begin{equation}\label{P2homo}
q''-2zq'+\nu q=0.
\end{equation}
We split the domain $\Omega=\mathbb{C}\setminus (-\infty,0]$ and choose a branch $$\log s=\log |s|+i\arg s, \arg s\in (-\pi,\pi).$$
Then we consider the contour $C$ in the $s$ plane: we start from $-\infty+i0^+$ and turn around $s=0$, then turn to $-\infty+i0^-$. Assume that the solution is of the form
$$ q(z)=\int_{C}e^{sz}Y(s)ds.
$$   
Plugging into the equation \eqref{P2homo}, we have(assuming legality for changing the differentiation and the integration)
$$ q''-2zq'+\nu q=\int_{C}Y(s)\left(s^2+\nu-2s\partial_s\right)e^{sz}ds.
$$
After Integration by part, we have that
\begin{equation*}
\begin{split}
q''-2zq'+\nu q=&\int_{C}[(s^2+\nu)Y(s)+2(sY(s))']e^{sz}ds
-\{2sY(s)e^{sz} \}_C\\
=&\int_{C}[(s^2+\nu+2)Y+2sY'(s)]e^{sz}dz
-\{2sY(s)e^{sz}\}_C,
\end{split}
\end{equation*}
where
$$ \{2sY(s)e^{sz} \}_C=[2sY(s)e^{sz}]^{s=-\infty+i0+}_{s=-\infty+i0^-},
$$
if both the limits exist.
We require that
$$ (s^2+\nu+2)Y+2sY'(s)=0,
$$
and a simple choice is
$$ Y(s)=e^{-\frac{s^2}{4}}s^{-\frac{\nu+2}{2}}=e^{-\frac{s^2}{4}-\frac{\nu+2}{2}\log s}.
$$
From the choice of the contour $C$, 
$ \{2sY(s)e^{sz}\}_C=0
$, for all $|\arg z|\leq \frac{\pi}{2}-\delta (\delta>0)$. Moreover, the interchange of the order of integration and differentiation is legality. Thus, we obtained one solution of the homogeneous equation of \eqref{P2homo}
\begin{equation}\label{q1} q_1(\nu,z):=\int_{C}e^{sz-\frac{s^2}{4}-\frac{\nu+2}{2}\log s}ds.
\end{equation}
For the convenience of the estimates, we will rewrite $q_1$ as an integral of another contour $C_+$ in the $s$ plane: Choosing a branch of $\log s=\log|s|+i\textrm{arg}(s)$, $\textrm{arg}(s)\in(0,2\pi)$ by splitting $\Omega_+:=\C\setminus [0,+\infty)$.
Then 
start from $+\infty+i0^+$ and turn around $s=0$, then turn to $+\infty+i0^-$. Changing $s$ to $-s$ in the expression \eqref{q1}, we have
\begin{equation}\label{q1-1}
q_1(\nu,z)=-\int_{C_+}(-s)^{-\frac{\nu+2}{2}}e^{-\frac{s^2}{4}-sz}ds=e^{-\frac{i\pi\nu}{2}}\int_{C_+}s^{-\frac{\nu+2}{2}}e^{-\frac{s^2}{4}-sz}ds.
\end{equation}
\begin{lem}[Watson's Lemma\cite{Watson}]\label{Watson}
	Let $g(s)$ be an analytic function in $0<\arg s<2\pi$. Assume that $g(s)$ satisfies the following asymptotics:
	$$ g(s)=O(e^{b|s|}),s\rightarrow\infty,\textrm{ for some }b\in\R,
	$$
	$$ sg(s)\sim_{s\rightarrow 0} \sum_{n=1}^{\infty}a_ns^{\lambda_n}
	$$
	in the sense that for all $N\in\N$,
	\begin{equation}\label{asymptotics-sense}
	 \left|sg(s)-\sum_{n=1}^N a_ns^{\lambda_n}\right|=o(|s|^{\lambda_N}), \quad s\rightarrow 0,
    \end{equation}
	where $0<\lambda_1<\lambda_2<\cdots$, and $\lambda_n\rightarrow \infty.$ Then the function
	$$ F(z)=\int_{C_+}e^{-sz}g(s)ds,\quad |\arg z|\leq \frac{\pi}{2}-\delta,\quad \delta>0
	$$
	satisfies the asymptotics:
	$$ F(z)\sim 2i\sum_{n=1}^{\infty}a_n\Gamma(\lambda_n)\sin(\lambda_n\pi)e^{i\lambda_n\pi}z^{-\lambda_n},\quad |z|\rightarrow\infty,
	$$
	in the same sense as \eqref{asymptotics-sense}.
\end{lem}
\begin{proof}
The proof can be found in \cite{Wang-Guo}. For the completeness, we present the proof for $z\in\R, z\rightarrow\infty$, which is enough for our need.
	From the asymptotic behaviour of $g(s)$ near $\infty$, for any $N$, there exists a bounded function $\sigma_N(s)>0$ with the property $\sigma_N(s)=o(1), s\rightarrow 0$,  such that
	\begin{equation}\label{error}
	\left|g(s)-\sum_{n=1}^Na_ns^{\lambda_n-1}\right|\leq \sigma_N(s)|s|^{\lambda_N-1}e^{b|s|}.
	\end{equation}
	By changing variables and shrinking the contour if necessary, we have
	$$ \int_{C_+}s^{\lambda_n-1}e^{-sz}ds=z^{-\lambda_n}\left(\int_{+\infty+i0^+}^{0+i0^+}+\int_{0+i0^-}^{+\infty+i0^-}\right)v^{\lambda_n-1}e^{-v}dv
	$$
	Since on $[0+i0^+, +\infty+i0^+)$, $v^{\lambda_n-1}=|v|^{\lambda_n-1}$, while on $[0+i0^-, +\infty+i0^-)$, $v^{\lambda_n-1}=|v|^{\lambda_n-1}e^{2\pi i\lambda_n}$, the integral above equals to
	$$ z^{-\lambda_n}\Gamma(\lambda_n)(e^{2\pi i\lambda_n}-1)=z^{-\lambda_n}\Gamma(\lambda_n)2i\sin(\lambda_n\pi)e^{i\lambda_n\pi}.
	$$
	It remains to estimate the remainder term
	$$ R_N(z)=\int_{C_+}e^{-sz}g(s)ds-\int_{C_+}\sum_{n=1}^Na_ns^{\lambda_n-1}e^{-sz}ds.
	$$
	From \eqref{error}, we have
	$$ |R_N(z)|\leq Cz^{-\lambda_N}\int_0^{\infty}e^{-u}\sigma_N\left(\frac{u}{z}\right)u^{\lambda_N-1}e^{\frac{bu}{z}}du.
	$$
	By dominating convergence, $|z^{\lambda_N}R_N(z)|=o(1)$, as $z\rightarrow+\infty$. This completes the proof of Lemma \ref{Watson}.
\end{proof}

\begin{cor}\label{y1-bound}
	Fix $0<\nu<1$. Then
	$$ q_1(\nu,z)\sim_{z\rightarrow+\infty} z^{\frac{\nu}{2}}\sum_{k=0}^{\infty}c_{2k}(\nu)z^{-2k},
	$$
	in the sense that for any $N\in\N$,
	$$ \left|z^{-\frac{\nu}{2}}q_1(\nu,z)-\sum_{k=0}^{N}c_{2k}(\nu)z^{-2k}\right|=o(|z|^{-2N}),\quad z\rightarrow +\infty,
	$$
uniformly with respect to $0<\nu<1$.
Moreover, $c_0(\nu)\rightarrow 2\pi i$ as $\nu\rightarrow 0$, and there exist $0<\nu_0<1$ $z_0>1$, such that for all $0<\nu<\nu_0, z\geq z_0$, $|q_1(\nu,z)|\geq \pi|z|^{\frac{\nu}{2}}$. Furthermore, we have
	\begin{equation}\label{derivative-1}
	\left|(\partial_{\nu}^{k}q_1)(\nu,z)\right|\leq C_k|z|^{\nu/2}|\log z|^k, \quad \left|(\partial_z^{k}q_1)(\nu,z)\right|\leq C_k|z|^{\frac{\nu}{2}-k}.
	\end{equation}
	for all $k\in\N^*$as $z\rightarrow +\infty$, where the constants $C_k$ are independent of $\nu$. 
\end{cor}

\begin{proof}
	Denote by $G(s)=e^{-\frac{i\pi\nu}{2}}s^{-\frac{\nu+2}{2}}e^{-\frac{s^2}{4}}$ which is bounded near $\infty$. Then by definition, $\displaystyle{q_1(\nu,z)=\int_{C_+}G(s)e^{-sz}ds }.$
	Note that  
	$$ s\left(G(s)-e^{-\frac{i\pi\nu}{2}}s^{-\frac{\nu+2}{2}}\right)\sim_{s\rightarrow 0}e^{-\frac{i\pi\nu}{2}}\sum_{k=1}^{\infty}\frac{(-1)^k}{4^k\cdot k!}s^{2k-\frac{\nu}{2}}.
	$$
	Applying Watson's Lemma \ref{Watson} for $g(s)=G(s)-e^{-\frac{i\pi\nu}{2}}s^{-\frac{\nu+2}{2}}$ with $\lambda_k=2k-\frac{\nu}{2}$, $a_k=e^{-\frac{i\pi\nu}{2}}\frac{(-1)^k}{4^k\cdot k!}$, we have
	$$ q_1(z)-e^{-\frac{i\pi\nu}{2}}\int_{C_+}s^{-\frac{\nu+2}{2}}e^{-sz}ds\sim 2iz^{\frac{\nu}{2}}\sin\left(\frac{\nu+2}{2}\pi\right)\sum_{k=1}^{\infty}\frac{(-1)^ke^{-i\pi\nu}}{4^k\cdot k!}\Gamma\left(2k-\frac{\nu}{2}\right)z^{-2k}.
	$$ 
	Note that the bound on the remainder terms is independent of $0<\nu<1$, since for $k\geq 1$, $\sin((\nu+2)\pi/2)\Gamma(2k-\nu/2)e^{-i\pi\nu}$ is bounded with respect to $\nu$. In particular, the coefficient $c_k(\nu)$ is bounded in $\nu$.
	
	By changing of variables and applying Cauchy's theorem, we have
	\begin{equation*}
	\begin{split}
	e^{-\frac{i\pi\nu}{2}}\int_{C_+}s^{-\frac{\nu+2}{2}}e^{-sz}ds=&e^{-\frac{i\pi\nu}{2}}z^{\frac{\nu}{2}}\int_{\widetilde{C}_+}u^{-\frac{\nu+2}{2}}e^{-u}du.
	\end{split}
	\end{equation*}
	where $\widetilde{C}_+$ is an enlargement of the contour $C_+$ so that $|s|\geq 1$ along $\widetilde{C}_+$. Hence 
	$$c_0(\nu)=e^{-\frac{i\pi\nu}{2}}\int_{\widetilde{C}_+}u^{-\frac{\nu+2}{2}}e^{-u}du.
	$$ 
	
	Next we show that $c_0\neq 0$ for sufficiently small $\nu$. Note that $\Gamma(s)$ is analytic for $\Re s>0$, therefore, by Cauchy's theorem, we have 
	\begin{equation}\label{Gamma}
	\Gamma(s)=\frac{e^{-i\pi s}}{2i\sin(\pi s)}\int_{\widetilde{C}_+}e^{-u}u^{s-1}du, \textrm{ for } \Re s>0.
	\end{equation}
	The right hand side and left hand side are both analytic functions for $s\in\C$ except for simple poles $s=0,-1,-2,\cdots$, therefore, \eqref{Gamma} holds true for all $s\in\C$, $s\notin -\N$.
	Therefore,
	$$ c_0(\nu)=e^{-\frac{i\pi\nu}{2}}\int_{C_+}u^{-\frac{\nu}{2}-1}e^{-u}du=-2ie^{-i\pi\nu}\sin\left(\frac{\pi\nu}{2}\right)\Gamma\left(-\frac{\nu}{2}\right)\neq 0
	$$
	for all $0<\nu<1$. Clearly, $c_0(\nu)\rightarrow 2\pi i$ as $\nu\rightarrow 0$, thanks to the formula
	$ \Gamma(s+1)=s\Gamma(s).
	$ Consequently, $|q_1(\nu,z)|$ has a uniform lower bound for small $\nu>0$ and large $z>1$. 
	
	Finally, we estimate the derivatives. From direct calculation and using Cauchy's theorem,
	$$ (\partial_{\nu}^kq_1)(\nu,z)=e^{-\frac{i\pi\nu}{2}}\int_{\widetilde{C}_+}P_k(\log s) s^{-\frac{\nu+2}{2}}e^{-\frac{s^2}{4}-sz}ds,
	$$
	where $P_k(\cdot)$ is a polynomial of degree $k$. Note that on $\widetilde{C}_+$ such that $|s|\geq 1$ around the origin, we deduce from Lemma \ref{Watson} that
	$$ \left|(\partial_{\nu}^{k}q_1)(\nu,z) \right|\leq C_k|z|^{\frac{\nu}{2}}|\log z|^k.
	$$
	For the derivatives on $z$, we calculate
	$$ (\partial_z^kq_1)(\nu,z)=-\int_{C_+}(-s)^{k-\frac{\nu+2}{2}}e^{-\frac{s^2}{4}-sz}ds.
	$$
	Using Watson's Lemma \ref{Watson} we have that 
	$|\partial_z^kq_1|\leq C_kz^{\frac{\nu}{2}-k} $.
	The proof of Corollary \ref{y1-bound} is complete.
\end{proof}

The Wronskian of \eqref{P2homo} satisfies $W'(z)=2zW$, thus another independent solution of \eqref{P2homo} is given by
$$ q_2(z)=q_1(z)\left[\int_{z_0}^z\frac{e^{w^2}}{(q_1(w))^2}dw+C\right]
$$
for $z>z_0$,
where $z_0>0$ is chosen as in  Corollary \ref{y1-bound}. Note that
\begin{equation}\label{Wronskian}
 W(z)=q_1q_2'-q_2q_1'=e^{z^2}.
\end{equation}
\begin{cor}\label{y2-bound}
	There exist $0<\nu_0<1, z_0>1$, such that for all $k\in\N$, $0<\nu<\nu_0$ and $z\geq z_0$, we have
	\begin{equation*}
	|q_2(\nu,z)|\leq Cz^{\frac{\nu}{2}-1}e^{z^2}, |\partial_{\nu}^{k}q_2(\nu,z)|\leq C_kz^{\frac{\nu}{2}-1}(\log z)^ke^{z^2}, |\partial_z^kq_2(\nu,z)|\leq C_k(z^{\frac{\nu}{2}-k-1}+z^{k-1-\frac{\nu}{2}})e^{z^2},
	\end{equation*}	
	where the constants $C_k$ are independent of $\nu$.
\end{cor}
\begin{proof}
 First we remark that we only need estimate the term $q_1(z)\int_{z_0}^z\frac{e^{w^2}}{q_1(w)^2}dw $. Let $\nu_0, z_0$ as in Corollary \ref{y1-bound}. 	Denote by 
 $ I(z)=\int_0^z e^{-(z^2-x^2)}dx.
 $
 From Lemma \ref{elementary}, $I(z)\sim\frac{1}{2z}$ for large $z$. Applying Corollary \ref{y1-bound}, we have
	$$ |q_2(\nu,z)|e^{-z^2}\leq \frac{Cz^{\nu/2}}{z_0^{\nu}}\int_{z_0}^ze^{x^2-z^2}dx\leq \frac{Cz^{\nu/2}}{z_0^{\nu}}I(z)\leq C(z_0)z^{\frac{\nu}{2}-1}.
	$$
	To estimate the derivatives, observe that
	for each $k\in\N^*$, $\partial_{\nu}^{k}(q_2)$ can be written as linear combinations of
	$$ \partial_{\nu}^{k_1}(q_1)\cdot \int_{z_0}^z e^{w^2}\partial_{\nu}^{k_2}(q_1^{-2})dw.
	$$ 
	where $k_1+k_2=k$. 	From Fa\`{a} di Bruno's formula, $\partial_{\nu}^{k_2}(q_1^{-2})$ is linear combinations of
	$$q_1^{-(2+m_1+m_2+\cdots+m_{k_2})}(\partial_{\nu}q_1)^{m_1}(\partial_{\nu}^2q_1)^{m_2}\cdots (\partial_{\nu}^{k_2}q_1)^{m_{k_2}},\textrm{ where } \sum_{j=1}^{k_2}jm_j=k_2.
	$$
	From Corollary \ref{y1-bound},
	$$ |\partial_{\nu}^{k_2}(q_2)^{-2}|\leq C_{k_2}z^{-\nu}|\log z|^{k_2}.
	$$
	Hence
	$$ |\partial^{k}_{\nu}q_2|\leq \sum_{k_1+k_2=k}C_kz^{\frac{\nu}{2}}(\log z)^{k_1}\int_{z_0}^ze^{w^2}(\log w)^{k_2}w^{-\nu}dw\leq C_k'z^{\frac{\nu}{2}}(\log z)^{k}e^{z^2}I(z).
	$$
	Applying Lemma \ref{elementary}, we obtain the desired bound. For $\partial_z^kq_2$, it can be written as linear combinations of
	$$\partial_z^kq_1\cdot \int_{z_0}^z\frac{e^{w^2}}{q_1(w)^2}dw  \textrm{ and } (\partial_{z}^{k_1}q_1)\cdot(\partial_z^{k_2}(q_1^{-2}))\cdot (\partial_z^{k-1-k_1-k_2}(e^{z^2}) ), k_1< k.
	$$
	The first term can be majorized by $C_k|z|^{\frac{\nu}{2}-k-1}$, thanks to Corollary \ref{y1-bound}. Again, from Fa\`{a} di Bruno's formula, $\partial_z^{k_2}(q_1^{-2})$ is the linear combination of the terms
	$$ q_1^{-(2+m_1+m_2+\cdots+m_{k_2})}(\partial_{z}q_1)^{m_1}(\partial_{z}^2q_1)^{m_2}\cdots (\partial_{z}^{k_2}q_1)^{m_{k_2}},\textrm{ where } \sum_{j=1}^{k_2}jm_j=k_2,
	$$
	and
	$$ |\partial_z^{k_2}(q_1^{-2})|\leq C_k|z|^{-\nu-k_2}.
	$$
	Using Corollary \ref{y1-bound}, we obtain that
	$$ |\partial_z^{k_1}q_1\cdot \partial_z^{k_2}(q_1^{-2})\cdot \partial_z^{k-k_1-k_2-1}(e^{z^2})|\leq C_k|z|^{\frac{\nu}{2}-k_1}\cdot |z|^{-\nu-k_2}\cdot |z|^{k-k_1-k_2-1}e^{z^2}\leq C_k|z|^{k-\frac{\nu}{2}-2(k_1+k_2)-1}e^{z^2}.
	$$
	The worst case is $k_1=k_2=0$. 
	This completes the proof of Corollary \ref{y2-bound}.
\end{proof}

\begin{cor}\label{bound-R2}
	As $z\rightarrow +\infty$, we have for all $k\in\N$, $$ |R_2(\nu,z)|\leq Cz^{\nu-1}e^{\frac{z^2}{2}},  |\partial_{\nu}^{k}R_2(\nu,z)|\leq C_kz^{\nu-1}|\log z|^{k}e^\frac{z^2}{2}, |\partial_{z}^{k}R_2(\nu,z)|\leq C_kz^{\nu+k-1}e^{\frac{z^2}{2}}.
	$$
	In particular, if $\nu\leq Aw^{1/2}e^{-w}$, then for all $1\leq |z|\leq w^{1/2}$, all $k\in\N$, we have
	$$ |\partial_z^kf(\nu,z)|\leq C_kw^{1/2}z^{k-1}e^{\frac{z^2}{2}},
	$$
	where $f(\nu,z)$ is the solution of \eqref{Hermite-1}.
\end{cor}
\begin{proof}
The solution of the inhomogeneous ODE \eqref{P2} is given by
	$$P_2(\nu,z)=-q_1(\nu,z)\int_0^z\frac{q_2(\nu,w)Q_2(w)}{W(w)}dw+q_2(z)\int_0^z\frac{q_1(\nu,w)Q_2(w)}{W(w)}dw.
	$$
	In particular,
	$$ P_2(\nu,z)=-q_1(\nu,z)\int_0^ze^{-w^2}q_2(\nu,w)Q_2(w)dw+q_2(\nu,w)\int_0^ze^{-w^2}q_1(\nu,w)Q_2(w)dw,
	$$
	where $Q_2(z)=-f_1(z)e^{\frac{z^2}{2}}$. From Proposition \ref{first-order},Corollary \ref{y1-bound} and Corollary \ref{y2-bound} and the fact that $I(z)\leq \frac{C}{z}$, we have
	$$ |P_2(\nu,z)|\leq Cz^{\nu-1}e^{z^2}.
	$$
	Recalling that $R_2(\nu,z)=e^{-\frac{z^2}{2}}P_2(\nu,z)$, we obtain that desired bound for $|R_2|$. 
	For the derivatives, we observe that $\partial_z^kP_2$ consists of linear combinations of 
	$$ (\partial_z^kq_1)\cdot \int_0^ze^{-\frac{w^2}{2}}f_1(w)q_2(\nu,w)dw,\quad  (\partial_z^kq_2)\cdot\int_0^ze^{-\frac{w^2}{2}}f_1(w)q_1(\nu,w)dw,
	$$
	and
	$$ (\partial_z^{k_1}q_1)\cdot \partial_z^{k_2-1}(e^{-\frac{z^2}{2}}f_1(z)q_2),\quad  (\partial_z^{k_1}q_2)\cdot \partial_z^{k_2-1}(e^{-\frac{z^2}{2}}f_1(z)q_1 )\textrm{ where $k_1+k_2=k$ if $k_2\geq 1$. }
	$$
	Thus from Proposition \ref{first-order}, Corollary \ref{y1-bound}, Corollary \ref{y2-bound}, the dominating part can be bounded by $C_k|z|^{\nu+k-1}e^{z^2}$. Since $R_2=e^{-\frac{z^2}{2}}P_2$, we obtain that 
	$ |\partial_z^kR_2|\leq C_k|z|^{\nu+k-1}e^{\frac{z^2}{2}}.
	$ The derivatives $\partial_{\nu}^kP_2$ consists of linear combinations of
	$$ (\partial_{\nu}^{k_1}q_1)\cdot \int_0^ze^{-\frac{w^2}{2}}f_1(w)(\partial_{\nu}^{k_2}q_2)(\nu,w)dw, \quad (\partial_{\nu}^{k_1}q_2)\cdot \int_0^ze^{-\frac{w^2}{2}}f_1(w)(\partial_{\nu}^{k_2}q_1)(\nu,w)dw,
	$$ 
	where $k_1+k_2=k$. Again, from Proposition \ref{first-order}, Corollary \ref{y1-bound}, Corollary \ref{y2-bound}, these terms are dominated by
	$|z|^{\nu-1}|\log z|^{k}e^{z^2}$. Thus $|\partial_{\nu}^kP_2|\leq C_k|z|^{\nu-1}|\log z|^{k}e^{\frac{z^2}{2}}$. The last assertion is the direct consequence of the estimates above and Proposition \ref{first-order}. The proof of Corollary \ref{bound-R2} is complete.
\end{proof}

\subsection{Estimate for the first eigenvalue and eigenfunction}
For fixed $w>0$, consider the function
$$ F(\nu,w):=f(\nu,\sqrt{w})=f_0(\sqrt{w})+\nu f_1(\sqrt{w})+\nu^2R_2(\nu,\sqrt{w}).
$$

\begin{prop}\label{fixed-point}
	There exists $w_0>0, A>0$, such that for all $w\geq w_0$, there exists a unique solution $\nu(w)\in (0,Aw^{1/2}e^{-w})$ of the equation
	$$ F(\nu(w),w)=0.
	$$
	Moreover, the application $w\mapsto \nu(w)$ is smooth.
\end{prop}

\begin{proof}
 For fixed $w>0$ large, we define the iteration scheme as follows:
	\begin{equation*}
	f_0(\sqrt{w})+\nu_1f_1(\sqrt{w})=0,\quad 
	 f_0(\sqrt{w})+\nu_{n+1}f_1(\sqrt{w})=-\nu_n^2R_2(\nu_n,\sqrt{w}),\quad n\geq 1.
	\end{equation*}
Since $f_1(z)$ never vanish for large $w$,  we have
	$$ \nu_{n+1}-\nu_n=\frac{G_2(\nu_{n-1},w)-G_2(\nu_n,w) }{f_1(\sqrt{w})}
	$$
	where $$
	G_2(\nu,w)=\nu^2R_2(\nu,\sqrt{w}).
	$$
	Applying mean-value theorem and Proposition \ref{first-order}, we obtain that
	$$ |\nu_{n+1}-\nu_{n}|\leq \frac{|\partial_{\nu}G_2(\theta\nu_{n-1}+(1-\theta)\nu_{n},w)|}{e^{w/2}c_1^{-1}w^{-1/2}}\cdot |\nu_{n}-\nu_{n-1}|.
	$$
	Assume that $\nu_j\leq Aw^{1/2}e^{-w}$, for all $j\leq n$, then from Corollary \ref{bound-R2}, the right hand side can be bounded by
	$$ 2Ac_1e^{-w}w^{\frac{\nu}{4}+1}|\nu_n-\nu_{n-1}|.
	$$
	Since $\nu_1=-f_0(\sqrt{w})/f_1(\sqrt{w})\in [c_1w^{1/2}e^{-w}, c_1^{-1}w^{1/2}e^{-w}]$, we have
	$$ \nu_{n+1}\leq \sum_{j=0}^{n}(2Ac_1e^{-w}w^{\frac{\nu}{4}+1})^j\nu_1\leq c_1^{-1}w^{1/2}e^{-w}(1-2Ac_1e^{-w}w^{\frac{\nu}{4}+1})^{-1}<Aw^{1/2}e^{-w},
	$$
	provided that $A>2c_1^{-1}$ and $w\geq w_0$ are large enough. Fixing such $A$ and $w_0$, by induction, we have that $\nu_n\leq Aw^{1/2}e^{-w}$ for all $n\in\N$. Now the absolute convergence of the sequence $\sum_{n=0}^{\infty}|\nu_{n+1}-\nu_{n}|$ (with the convention $\nu_0=0$) implies the existence of the limit
	$$ \nu(w)=\lim_{n\rightarrow\infty}\nu_n(w) \in[0,Aw^{1/2}e^{-w}].
	$$
 Next we show that $F(\nu(w),w)=0$. From the continuity of $F(\cdot,w)$, we have $F(\nu_n,w)\rightarrow F(\nu(w),w)$ as $n\rightarrow\infty$. By definition, we write
	$$ f_0(\sqrt{w})+\nu_nf_1(\sqrt{w})=-F(\nu_{n-1},w)+f_0(\sqrt{w})+\nu_{n-1}f_1(\sqrt{w}).
	$$ 
	Passing to the limit for both sides, we have $F(\nu(w),w)=0$. The uniqueness follows from similar argument, and the smoothness of the mapping $w\mapsto \nu(w)$ is guranteed by the inverse function theorem. This completes the proof of Proposition \ref{fixed-point}.
\end{proof}

\begin{proof}[Proof of Proposition \ref{estimate:eigenvalue}]
	From Proposition \ref{first-order}, Corollary \ref{bound-R2}, we know that 
	$$ |f_1^{(k)}(\sqrt{w})|\leq C_kw^{\frac{k-1}{2}}e^{\frac{w}{2}}, |\partial_{\nu}^kR_2(\nu,\sqrt{w})|\leq C_kw^{\frac{\nu-1}{2}}|\log w|^ke^{\frac{w}{2}}, |\partial_z^kR_2(\nu,\sqrt{w})|\leq C_kw^{\frac{\nu+k-1}{2}}e^{\frac{w}{2}}.
	$$ 
	We remark that the exact power of $w$ is of no importance, since the dominated factors are always exponentially in $w$. To simplify the notation, we will replace all the power of $w$ as well as $\log w$ by some polynomials $P_k$ in variable $\sqrt{w}$, which may change line by line in the calculations below. 
	
	The first estimate of (1) is a direct consequence of Proposition \ref{fixed-point}. To estimate the derivative $\nu^{(k)}(w)$, we first take one derivative in $w$ to the equation
	$$ F(\nu(w),w)=f_0(\sqrt{w})+\nu(w) f_1(\sqrt{w})+\nu(w)^2R_2(\nu(w),\sqrt{w})=0.
	$$
	Since $\nu(w)\leq Cw^{1/2}e^{-w}$, the dominating coefficient of $\nu'(w)$ is $f_1(\sqrt{w})$, since $f_1$ and $R_2$ have the same $e^{w/2}$ factor. Moreover, the $|\partial_wf_0(\sqrt{w})|$ is always much larger than $|2\nu(w)R_2(\nu(w),\sqrt{w})+\nu(w)^2\partial_w(R_2(\nu(w),\sqrt{w}) )|$. Therefore, $|\nu'(w)|\leq C_1w^{A_1}e^{-w}$ holds true with some power $A_1$. Taking $k$ derivatives in $w$, the dominating coefficient of $\nu^{(k)}(w)$ is still $f_1(\sqrt{w})$ while $|\partial_{w}^k(f_0(\sqrt{w}))|$ is always much larger than all other terms, therefore, $|\nu^{(k)}(w)|\leq C_kw^{A_k}e^{-w}$. Then (2) and (3) can be deduced from the bound of $\nu^{(k)}(w)$ easily.
	
	Finally, we estimate $p$ and $\partial_w^{k}p$. We first claim that
	$$ \|f(\nu(w),\sqrt{w}\cdot )\|_{L^2((-1,1))}^2=w^{-1/2}\int_{-w^{1/2}}^{w^{1/2}}|f(\nu(w),z )|^2dz\sim w^{-1/2}.
	$$
Indeed, the upper bound follows directly from Proposition \ref{first-order} and Corollary \ref{bound-R2}. For the lower bound, note that for $|z|<w^{1/2}/2$, $\nu<Aw^{1/2}e^{-w}$, Proposition \ref{first-order} and Corollary \ref{bound-R2} yield 
$$ |f(\nu,z)|\geq e^{-\frac{z^2}{2}}-Cw^{1/2}e^{-w}z^{-1/2}e^{z^2/2}\geq  e^{-\frac{z^2}{2}}\left(1-Cw^{1/2}z^{-1/2}e^{-(w-z^2)}\right)>e^{-\frac{z^2}{2}}/2,
$$
if $w$ is large enough. Hence $\int_{-w^{1/2}}^{w^{1/2}}|f(\nu,z)|^2dz$ is uniformly bounded from below. 

Assume that $|x|\leq\frac{1}{2}$. Observe that
$$ |\nu f_1(\sqrt{w}x)+\nu^2R_2(\nu,\sqrt{w}x)|\leq Cw^{A_1}e^{-\frac{wx^2}{2}}e^{-(w-wx^2)}\leq Cw^{A_1}e^{-\frac{3w}{4}}e^{-\frac{wx^2}{2}},
$$
hence $|p(w,x)|\sim w^{1/4}e^{-\frac{wx^2}{2}}$.
Next we estimate the derivatives of $p(w,x)$. Direct calculation shows $$\partial_w(\|f(\nu(w),\sqrt{w}\cdot )\|_{L^2((-1,1))}^2)=O(w^{-3/2}).$$ 
Consequently, the absolute value of the second term of
$$ \partial_wp(w,x)=\frac{\partial_w(f(\nu(w),\sqrt{w}x)) }{\|f(\nu(w),\sqrt{w}\cdot)\|_{L^2((-1,1))}}-\frac{f(\nu(w),\sqrt{w}x)\partial_w(\|f(\nu(w),\sqrt{w}\cdot)\|_{L_x^2}^2 ) }{2\|f(\nu(w),\sqrt{w}\cdot)\|_{L^2((-1,1))}^3 },
$$
can be bounded by $O(w^{-3/4})$. 
To control the other term, we write
\begin{equation*}
\begin{split}
\partial_w(f(\nu(w),\sqrt{w}x) )=&-\frac{x^2}{2}e^{-\frac{wx^2}{2}}+\nu'f_1(\sqrt{w}x)+2\nu\nu'R_2(\nu,\sqrt{w}x)\\
+&\nu\frac{x}{2\sqrt{w}}f_1'(\sqrt{w}x)+\frac{\nu^2 x}{2\sqrt{w}}(\partial_zR_2)(\nu,\sqrt{w}x)+\nu^2\nu'(\partial_{\nu}R_2)(\nu,\sqrt{w}x).
\end{split}
\end{equation*}
From Proposition \ref{first-order} and Corollary \ref{bound-R2}, all the terms except for the first term can be bounded by $O(w^{A}e^{-\frac{3w}{4}})$. The first term $x^2e^{-wx^2/2}=(wx^2)e^{-wx^2/2}=O(w^{-1})$. Therefore, 
$ |\partial_wp(w,x)|=O(w^{-3/4}).
$
Taking one more derivatives, one easily verifies that $|\partial_w^2p(w,x)|=O(w^{-\frac{7}{4}})$.  

Finally, for $\frac{1}{2}<|x_0|\leq 1$,
$$ |p(w,x)|\leq Cw^{1/4}|f(\nu(w),\sqrt{w}x)|\leq w^{1/4}e^{-\frac{w}{8}}+Cw^{A}e^{-\frac{w}{2}}\leq e^{-\frac{w}{10}},
$$ 
provided that $w$ is large enough. This completes the proof of Proposition \ref{estimate:eigenvalue}.
\end{proof}

\appendix

\section{Non-observability for the Grushin wave equation}

In this section, we consider the Grushin-wave equation
\begin{align}\label{wave}
\begin{cases}
 \partial_t^2u-\Delta_Gu=0,\quad (t,x,y)\in\R_t\times \Omega \\
 (u,\partial_tu)|_{t=0}=(u_0,u_1)\in \dot{H}_{0,G}^1(\Omega)\times L^2(\Omega).
\end{cases}
\end{align}
Recall that the semi-norm of $\dot{H}_{0,G}^1$ is given in \eqref{H1G}.

We will show that for any time $T>0$, the following observability 
\begin{align}\label{ob:wave}
\int_{\Omega}(|\nabla_Gu_0(x,y)|^2+|u_1(x,y)|^2)dxdy\leq C_{T,\omega}\int_0^T\int_{\Omega}|\partial_tu(t,x,y)|^2dxdydt
\end{align}
 cannot hold, if the observation region $\omega$ does not contain the whole circle $\{(x,y):x=0\}=\{0\}_x\times\T_y$:
\begin{prop}\label{non-observabilitywave}	
Let $T>0$ and $\omega\subset\Omega$ be a open subset such that $$\omega\cap \{(x,y)\in\Omega: x=0 \}\neq \{(x,y)\in\Omega: x=0 \}.$$ 
Then there exists a sequence of solutions $(u_n)_{n\in\N}$ of \eqref{wave} satisfying 
$$ \|\nabla_Gu_{n}(0,\cdot)\|_{L^2(\Omega)}^2+\|\partial_tu_n(0,\cdot)\|_{L^2(\Omega)}\gtrsim 1
$$ 
for all large $n$
and
$$ \lim_{n\rightarrow\infty}\int_0^T\|\partial_tu_n(t,\cdot)\|_{L^2(\Omega)}^2dt=0.
$$
In particular, the observability estimate \eqref{ob:wave} is untrue.
\end{prop}

\begin{proof}
Without loss of generality, we may assume that a small neighborhood $U$ of $(0,0)$ does not intersect with the closure of $\omega$.
Pick a sequence $(h_n)_{n\in\N}$ such that $h_n\rightarrow 0$ as $n\rightarrow\infty$ and consider the following sequence of approximate solutions:
$$ v_{n}(t,x,y)=h_n^{1/8}\exp\big(i\big(\frac{y}{h_n}+\frac{t}{h_n^{1/2}}\big) \big)\cdot \exp\big(-\big(\frac{x^2}{2h_n}+\frac{y^4}{4h_n} \big)\big).
$$
Direct computation yields
$$ (\partial_t^2-\Delta_G)v_n=f_n,
$$
with
$$ f_n(t,x,y)=\big[\frac{3x^2y^2}{h_n}+\frac{x^2}{h_n^2}(2iy^3-y^6) \big]v_n(t,x,y).
$$
One verifies that for any fixed $t\in\R$,
$$ \|\nabla_Gv_n(t,\cdot)\|_{L^2(\R^2)}\sim 1, \; \|\partial_tv_n(t,\cdot)\|_{L^2(\R^2)}\sim 1,
$$
where the implicit constants are independent of $t$. Moreover, for large $n$, 
$$ \|f_h(t,\cdot)\|_{L^2(\R^2)}=O(h_n^{\frac{1}{4}}),\; \int_0^T\|\partial_tv_n(t)\|_{L^2(\omega)}^2dt\lesssim C_{N_0}Th^{N_0}
$$
for any $N_0\in\N$.

Next we fix a cutoff function $\chi(x,y)$ with support close to $(0,0)$ and $\chi\equiv 1$ on $U$. Denote by $w_n=\chi\cdot v_n$. Since $v_n(t,x,y)$ is concentrated on $|y|\leq O(h_n^{1/4}), |x|\leq O(h_n^{1/2})$, for sufficiently large $n$, we still have
\begin{align}\label{initial} 
 \|\nabla_Gw_n(t,\cdot)\|_{L^2(\Omega)}\sim 1,\; \|\partial_tw_n(t,\cdot)\|_{L^2(\Omega)}\sim 1.
\end{align}
Furthermore, $w_n\in \dot{H}_{0,G}^1(\Omega)$ and it satisfies the equation
$$ (\partial_t^2-\Delta_G)w_n=\widetilde{f}_n,
$$
with $$\|\widetilde{f}_n(t,)\|_{L^2(\Omega)}=O(h_n^{1/4}),\; \int_0^T\|\partial_tw_n(t)\|_{L^2(\omega)}^2dt\lesssim C_{N_0}Th^{N_0} $$ for any $N_0\in\N$. 
Finally, let $u_n$ be the solution of \eqref{wave} with the initial data $(w_n(0,\cdot),\partial_tw_n(0,\cdot))$, by the energy estimate of $u_n-w_n$, we have for any $t>0$,
$$ \|\nabla_G(u_n-w_n)(t)\|_{L^2(\Omega)}^2+\|\partial_t(u_n-w_n)(t)\|_{L^2(\Omega)}^2\leq \int_0^t\|\partial_t(u_n-w_n)(t)\|_{L^2(\Omega)}\|\widetilde{f_n}(t)\|_{L^2(\Omega)}dt,
$$
which implies that 
$$ \sup_{t\in[0,T]}\|\partial_t(u_n-w_n)(t)\|_{L^2(\Omega)}=O(h_n^{\frac{1}{4}}).
$$
Thus
$$ \int_0^T\|\partial_tu_n(t)\|_{L^2(\omega)}^2\lesssim \int_0^T\|\partial_tw_n(t)\|_{L^2(\omega)}^2dt+\int_0^T\|\partial_t(u_n-w_n)(t)\|_{L^2(\Omega)}^2dt=O(h_n^{1/2}). 
$$
In view of $(u_n(0,\cdot),\partial_tu_n(0,\cdot) )=(w_n(0,\cdot),\partial_tw_n(0,\cdot) ) $ and \eqref{initial}, the proof of Proposition \ref{non-observabilitywave} is complete.

\end{proof}

\begin{rem}
In \cite{Let}, the author constructed approximate solutions (Gaussian beams) concentrated on the rays in the elliptic region. These rays remain for a long time close to the subelliptic region, that leads to the proof of the non-observability for the subelliptic wave equation. In our special setting, we directly construct approximate solutions that are concentrated on the subelliptic region.

\end{rem}


\end{document}